\documentclass{amsart}
\usepackage{amsthm,amstext,amsmath,amscd,amssymb,latexsym}
\usepackage{color}
\usepackage[matrix,arrow]{xy}
\usepackage{amsfonts,enumerate,array}
\usepackage[colorlinks=true]{hyperref}
\usepackage{todonotes}
\usepackage{verbatim}

\newcommand{\red}[1]{\textcolor{red}{#1}}

\newcommand{\CC}{{\mathbb C}}
\newcommand{\PP}{{\mathbb P}}

\newcommand{\OO}{{\mathcal O}}
\newcommand{\LL}{{\mathcal L}}

\newcommand{\II}{{\mathcal I}}

\newcommand{\Bs}{{\rm{Bs}}}

\newcommand{\hh}{{\rm{h}}}
\newcommand{\HH}{{\rm{H}}}
\newcommand{\ls}{{\mathcal{L}}}
\newcommand{\ws}{{\mathcal{W}}}

\DeclareMathOperator{\HHilb}{Hilb}

\DeclareMathOperator{\vdim}{vdim}
\DeclareMathOperator{\edim}{edim}
\DeclareMathOperator{\ldim}{ldim}

\DeclareMathOperator{\Pic}{Pic}
\DeclareMathOperator{\Res}{Res}

\newcommand{\paper}{: \begin{it}}
\newcommand{\jour }{, \end{it}}

\newtheorem{theorem}{Theorem}[section]
\newtheorem{lemma}[theorem]{Lemma}
\newtheorem{proposition}[theorem]{Proposition}
\newtheorem{corollary}[theorem]{Corollary}
\newtheorem{conjecture}[theorem]{Conjecture}

\newtheorem{question}[theorem]{Question}
\theoremstyle{definition}
\newtheorem{definition}[theorem]{Definition}

\newtheorem{example}[theorem]{Example}

\theoremstyle{remark}
\newtheorem{remark}[theorem]{Remark}

\numberwithin{equation}{section}

\begin{document}

\title{On a notion of speciality of linear systems in $\PP^n$}

\author{Maria Chiara Brambilla}
\email{{\tt brambilla@dipmat.univpm.it}}
\address{Universit\`a Politecnica delle Marche, 
via Brecce Bianche, I-60131 Ancona, Italy}

\author{Olivia Dumitrescu}
\email{{\tt dolivia@math.ucdavis.edu}}
\address{University of California, Davis, Department of Mathematics, MSB 2107, Davis, CA 95616}

\author{Elisa Postinghel}
\email{{\tt epostinghel@impan.pl}}
\address{Institute of Mathematics of the Polish Academy
of Sciences, ul. \'Sniadeckich 8, P.O. Box 21, 00-956 Warszawa, Poland}

\thanks{The first author is partially supported by Italian MIUR funds.
The second author is member of ``Simion Stoilow'' Institute of Mathematics 
of the Romanian Academy (http://www.imar.ro/).
The third author was partially supported by Marie-Curie IT Network SAGA, [FP7/2007-2013] grant agreement PITN-GA-
2008-214584. All authors were partially supported by Institut Mittag-Leffler.}

\keywords{Linear systems, Fat points, Base locus, Linear speciality, Effective cone}

\subjclass[2010]{Primary: 14C20. Secondary: 14J70, 14C17.}

\begin{abstract}
Given a linear system  in $\PP^n$ with assigned multiple general points we compute the cohomology groups
of its strict transforms  via the blow-up of its linear base locus. 
This leads us to give a new definition of expected dimension of a linear system, which takes into account the contribution of the linear base locus, and thus to introduce the notion of linear speciality. We investigate such a notion giving sufficient conditions for a linear system to be linearly non-special for arbitrary number of points, and necessary conditions for small numbers of points.
\end{abstract}

\maketitle

\section{Introduction}

The study of linear systems of hypersurfaces in
complex projective spaces with finitely many assigned base points of given multiplicities 
is a fundamental problem in algebraic geometry, related to polynomial interpolation in several variables,
to the Waring problem for polynomials and to the classification of defective 
higher secant varieties to projective varieties.

Let $\LL=\LL_{n,d}(m_1,\ldots,m_s)$ be the linear system of 
hypersurfaces of degree $d$ in $\PP^n$ passing through a general union of $s$ points with multiplicities respectively 
$m_1,\ldots,m_s$.
The {\em virtual dimension} of $\LL$ is
$$\vdim(\LL)=\binom{n+d}{n}-\sum_{i=1}^s\binom{n+m_i-1}{n}-1,$$
the {\em expected dimension} of $\LL$ is $\edim(\LL)=\max(\vdim(\LL),-1)$.
The dimension of $\LL$ is upper-semicontinuous in the position of the points in $\PP^n$;
 it achieves its minimum value when they are in general position.  The inequality
$\dim(\LL)\ge\edim(\LL)$ is  always satisfied.
If the conditions imposed by the assigned points are not linearly independent, then the actual dimension of $\LL$ is strictly
 greater that the expected one: 
in that case we say that $\LL$  (or a divisor $D$ in $\LL$) is \emph{special}. Otherwise, if the actual and the expected dimension coincide, 
we say that $\LL$ is \emph{non-special}.

Recall that if $Z$ is a collection of general fat points in $\PP^n$ of multiplicities $m_1,\ldots,m_s$,
then the sheaf  associated to the linear system $\LL$ is  $\OO_{\PP^n}(d)\otimes\II_Z$. For this reason, 
by abuse of notation, we will use the same letter $\LL$ to denote such a sheaf, when no confusion arises.
From the restriction exact sequence of sheaves
$$
0\rightarrow \LL\rightarrow \OO_{\PP^n}(d)\rightarrow{\OO}_{Z}\to0,
$$
taking cohomology we get 
$$
0\rightarrow H^0(\PP^n,\LL)\rightarrow H^0(\PP^n,\OO_{\PP^n}(d))\rightarrow  H^0(Z,{\OO}_{Z})\rightarrow H^1(\PP^n,\LL)\rightarrow 0,
$$
being $h^1(\PP^n,\OO_{\PP^n}(d))=0$. Moreover we obtain $h^i(\PP^n,\LL)=0$, for all $i\ge 2$. 
Thus $\LL$ is non-special if and only if 
$$
h^0(\PP^n,\LL)\cdot h^1(\PP^n,\LL) = 0.
$$

The problem of classifying  special linear systems attracted the attention of many researchers in the last century.
In the case of the plane, the well-known Segre-Harbourne-Gimigliano-Hirschowitz conjecture describes all 
special linear systems, but even in this case, in spite of many partial results 
(see e.g.\ \cite{Ciliberto}, \cite{CHMR} and the reference therein), the conjecture is still open. 
In the case of $\PP^3$, there is an analogous conjecture formulated by Laface and Ugaglia, 
see Section \ref{section6.2} for more details.

In the direction of extending such conjectures to $\PP^n$, and possibly to 
other projective varieties, we start with a very natural and general question.
\begin{question}\label{question}
Consider any non-empty linear system $\LL$ in $\PP^{n}$ and a divisor $D\in \LL$.
Denote by $\widetilde{D}$ the strict transform of $D$  in 
the blow-up $X$ of $\PP^n$ along the base locus of $\LL$. 
Is $\widetilde{D}$ non-special, namely, does $h^i(X,\mathcal{O}_{X}(\widetilde{D}))$ vanish for all $i\ge 1$?
\end{question}

In order to answer this question one has to tackle two problems:
the first one is to describe the base locus of a linear system, the second one is to understand the contribution given by each component of the base locus to the speciality of the linear system.

In this paper we start considering the case when the base locus of the linear system is linear, that is given by the union of linear subspaces of $\PP^n$  of dimensions $1\le r\le n-1$. 
In this case, it is possible to give arithmetical criteria, 
see Section \ref{base-locus-section}, 
which tell when a linear subspace
is contained in the base locus and with which multiplicity. 
Moreover  the contribution to the speciality of the linear system is easy to compute. 

In order to show this, let us consider first an example.
The well-known Alexander-Hirschowitz theorem states that a linear system in $\PP^n$ with only double points is special
 either if the degree is $2$ and the number of points is $2\le s\le n$, or if the linear system is one of 
the four exceptional cases, see e.g. \cite{ale-hirsch},\cite{ale-hirsch2} for more details.
The first case is easily understood,
indeed the linear system $\LL=\LL_{n,2}(2^s)$ satisfies
$\vdim(\LL)={{n+2}\choose2}-s(n+1)-1$.
On the other hand it is easy to see that any hypersurface in $\LL$ is a quadric cone with vertex the linear subspace $\PP^{s-1}$,
hence $\dim(\LL)=\binom{n-s+2}{2}-1$, which is the dimension
of the complete linear system of quadric hypersurfaces in $\PP^{n-s}$. 
If $s=2$, then $\dim(\LL)-\vdim(\LL)=1$; 
therefore we may conjecture that the double line 
which is contained in the base locus gives an obstruction of $1$.
Similarly if $s=3$, then $\dim(\LL)-\vdim(\LL)=3$ and the presence of three double lines in the base locus, each of them 
contribuiting by $1$, would highlight the same phenomenon;
it is interesting to note that the plane spanned by the three base
points, which is doubly contained in the base locus as well, does not give any contribution.

The example of quadrics can be extended to any linear system with degree $d$ and 
points of multiplicities $d$, see Section \ref{cones-section}. 
More generally, any linear subspace 
$\PP^r$ which is contained in the base locus with multiplicity $k\ge r+1$ gives
a contribution to the speciality of $\LL$,
 which depends on $r$ and $k$.

This suggests a new extended definition of expected 
dimension which takes in account such {\it linear obstructions}. We call 
this new notion {\em linear  expected dimension} of $\LL$ and we denote it by $\ldim(\LL)$
(see Definition \ref{new-definition} for the precise formulation).
Consequently, we  say that $\LL$ is {\em linearly special} if $\dim(\LL)\neq\ldim(\LL)$.
In this sense the quadrics with double points are no longer special.

The linear expected dimension is meant to be a refined version of the expected dimension. 
We surmise   
to have $\edim(\ls)\le \ldim(\ls)\le \dim(\ls)$.
The second inequality is in fact also predicted by the so-called {\it weak Fr\"oberg-Iarrobino conjecture},
see Section \ref{section6.1} and \cite{Chandler} for more details.

In this paper we investigate the notion of linear speciality and we 
study conditions for a linear system to be not linearly special. 

The first step in this direction, which also allows to partially answer 
Question \ref{question}, 
is to give a detailed description of the cohomology of the strict transform 
of $\LL$ with respect to the blow-up of $\PP^n$ along the linear base locus. 
In particular we will show that any $r$-dimensional cycle in the base locus 
of $\LL$ gives a contribution at the level of the $r$-th cohomology group 
of the strict transform of $\LL$, after blowing up all linear 
base cycles of dimension at most $r-1$. 
See Theorem \ref{monster-theorem} for a more precise statement. 

A consequence of this result is that
a non-empty linear system $\LL_{n,d}(m_1,\ldots,m_s)$
in $\PP^n$ is linearly non-special as soon as $s\le n+2$, 
(see Corollary \ref{corollary-monster}).

In our opinion the description of the cohomology groups of the strict transform of the linear system is 
our main original contribution to this subject. 
See Section \ref{section6.1}
for a discussion on the connections between our approach
and the Fr\"oberg-Iarrobino conjecture, and with a comparison with previous results by Chandler.

When the points are more than $n+2$, we are able to give a sufficient condition for
 a linear systems to be linearly non-special,
see Theorem \ref{theorem-n+3} for the precise statement.
In this case, it is easy to find linear systems which are linearly special. 
For instance the speciality might be given by rational normal curves or by quadric hypersurfaces
contained in the base locus. 
In Section \ref{section6.2} we give a short account on future directions of our research.

We would like to point out, finally, that our results have also an interesting interpretation in the
setting of moduli space of stable rational curves with marked points,  
see Section \ref{section6.3} for more details.
  
The paper is organized as follows.
In Section \ref{base-locus-section} we describe the linear components of the base locus of any linear system.
In Section \ref{cones-section} we give the definition of linear expected dimension and of linear speciality, and we discuss some basic examples.
Sections \ref{toric-case-section} and \ref{n+3-section} contain the main results of the paper; precisely Theorem \ref{monster-theorem} and 
Corollary \ref{corollary-monster} are
devoted to the case $s\le n+2$, while Theorem \ref{theorem-n+3} concerns the case $s\ge n+3$.
In Section \ref{final-section}
we discuss the links with other approaches, interesting connections
and future directions.

\subsection{Acknowledgements} The authors would like to thank the Institut Mittag-Leffler, Stockholm (Sweden),
and the Universit\`a Politecnica delle Marche, Ancona (Italy), 
for the hospitality during their stay that promoted and made this collaboration possible. 
We would also like to express our gratitude to Edoardo Ballico, Ana-Maria Castravet, Renzo Cavalieri, Igor Dolgachev, Antonio Laface, 
Giorgio Ottaviani for their remarks and comments regarding this project.
We thank the anonymous referee for useful comments and suggestions.

\section{The linear components of the base locus}
\label{base-locus-section}

Let  $\LL:=\LL_{n,d}(m_{1},...,m_{s})$ be a linear system with multiple base points supported 
at general points $p_1,\dots,p_s\in \PP^n$. Let moreover 
$I(r)\subseteq\{1,\dots,s\}$ be any  multi-index
   of length $|I(r)|=r+1$, for $0\le r\le \min(n,s)-1$. We denote 
by $L_{I(r)}$ the unique $r$-linear cycle through the points $p_i$, for $i \in I(r)$.
We introduce the following notation:
\begin{equation}\label{definition-of-k}
k_{I(r)}:=\max\left(\left(\sum_{i\in I(r)}m_i\right)-rd, 0\right).
\end{equation}

The following lemma is equivalent to \cite[Cor. 5.2]{Chandler},  
however we include here the proof for the sake of completeness.
\begin{lemma}[Linear Base Locus Lemma]
\label{base-locus-lemma}
Let $\LL:=\LL_{n,d}(m_{1},...,m_{s})$ be a non-empty linear system. In the notation of above, assume that $0\le r\le n-1$ and
$k_{I(r)}>0$. Then $\LL$ contains in its base locus the cycle $L_{I(r)}$ 
with multiplicity at least $k_{I(r)}$.
\end{lemma}
\begin{proof}
We use induction on $r$. If $r=1$ the statement follows from B\'ezout's theorem. If $r>1$ consider the $r$-cycle
$L_{I(r)}$ spanned by $p_1,\dots, p_r,p_{r+1}$ as the cone over the $(r-1)$-cycle $L_{I(r-1)}$ spanned 
by $p_1,\dots,p_r$ and with vertex $p_{r+1}$. Notice that $k_{I(r-1)}\geq k_{I(r)}$, being $d\ge m_{r+1}$.
 By the inductive hypothesis, any point $q\in L_{I(r-1)}$ 
 is contained in the base locus of $\LL$  with multiplicity at least $k_{I(r-1)}=m_{1}+...+m_{r}-(r-1)d>0$. 
Therefore the line spanned by $q$ and $p_{r+1}$ is contained in the base locus 
with multiplicity at least $k_{I(r-1)}+m_{r+1}-d=k_{I(r)}$. 
\end{proof}

Given the general points $p_1,\dots,p_s$ in $\PP^n$, we denote by
\[\pi_{(0)}^n:X_{(0)}^n\to\PP^n\]
the blow-up of $\PP^n$ at $p_1,\dots,p_s$, with $E_1,\dots,E_s$ exceptional divisors.
The index $(0)$ indicates that the space $\PP^n$ is blown-up at $0$-dimensional 
schemes; 
in the same way  
$X^n_{(r)}$ will denote the $n$-dimensional projective space
blown-up along arrangements of linear cycles  of dimension $\le r$ spanned by the points $p_i$
(see Section \ref{section4.1} for more details).

The Picard group of $X_{(0)}^n$ is spanned by the class $H$ of a general hyperplane  and the
 exceptional divisors $E_{i}$, $i=1,\dots,s$. As in \cite{castravet-tevelev} we introduce a symmetric bilinear form 
on the blown up $\PP^{n}$ by  
$$\langle E_{i}, E_{j}\rangle=-\delta_{i,j}, \ \langle E_{i}, H\rangle=0, \ \langle H,H\rangle= n-1.$$

We recall that the standard Cremona transformation along  the coordinate points of $\PP^{n}$ is defined to be the birational map
$$[x_0, \dots, x_{n}]\rightarrow[x_{1}\dots x_{n}, \dots, x_{0}\dots x_{n-1}].$$
This map is given by the linear system of hypersurfaces of degree $n$
with multiplicity $n-1$ at each of the $n+1$ coordinate points.  
Moreover it induces an automorphism of the Picard group of the  blow-up $X_{(0)}^{n}$ 
at $s\geq n+1$ points by sending any divisor
$dH-m_{1}E_{1}-\dots-m_{s}E_{s}$
to
$$(d-c)H-(m_1-c)E_{1}-\dots-(m_{n+1}-c)E_{n+1}-m_{n+2}E_{n+2}-\dots-m_{s}E_{s},$$
where $c=m_{1}+\dots+m_{n+1}-(n-1)d$ and the first $n+1$ points are chosen to be the coordinate points of $\PP^{n}$.

Let $W$ denote the Weyl group of the blow-up $X_{(0)}^{n}$ at $s$ points. 
Recall that every element of 
 $W$ corresponds to a birational map of $\PP^{n}$ lying in the group generated by standard 
Cremona transformations and projective automorphisms of $\PP^{n}$. 
For more information on the properties of the Cremona transformations and the Weyl group
see  \cite[Section 2.3]{CDD} or \cite{Dolgachev, Dumnicki}.

The study of the effective cone of a linear system is extremely difficult in general, 
for instance in the case $n=2$ and $s\geq 10$ Nagata's conjecture as well as the computation of the effective 
cone $\overline{M_{0,n}}$ are still open problems (see also Section \ref{section6.3}).
In the following lemma we describe the effective cone of a linear system with a small number of points and 
we briefly sketch the proof, see also \cite[Lemma 4.8]{CDD}.

\begin{lemma}[Effectivity Lemma]\label{effectivity}
A linear system $\LL=\LL_{n,d}(m_1,\dots,m_s)$ is non-empty if 
\begin{equation}\label{condition-mult}
m_i\le d \quad \forall i\quad \mbox{ and }\quad \sum_{i=1}^s m_i\le nd.
\end{equation} 
Moreover if $s\leq n+2$, then $\LL$ is non-empty 
if and only if conditions \eqref{condition-mult} are satisfied.
In particular the faces of the effective cone of the blow-up of $\PP^{n}$
 at $s$ points are given by 
$\{d= m_1\},\dots,\{ d= m_{s}\}$
if $s\le n$ and by 
$\{d= m_1\},\dots,\{ d= m_{s}\},\{ nd=\sum_{i=1}^s m_{i}\}$
if $s=m+1,n+2$. 
\end{lemma}

\begin{proof}
By \cite[Lemma 4.24]{castravet-tevelev}, we know that 
conditions \eqref{condition-mult} imply that $\LL$ is not empty. 

Conversely, consider an effective divisor $D\in \LL$ and note that
we must have $d\geq m_1,\dots, d\geq m_{s}$. 
Assume now by contradiction that $nd<\sum_{i=1}^s m_{i}.$ 

If $s=n+1$, then since $d\geq m_i$, by 
Lemma \ref{base-locus-lemma}
we get that each of the $n+1$ hyperplanes $H_i$, spanned by all but the $i$-th point, 
is in the base locus at least
$k(i)= m_1+\dots +\widehat{m_i}+\dots +m_{n+1}-(n-1)d>0$ times.
We obtain $D=\sum_{i=1}^{ n+1} k(i)H_{i}+\Res$, where $\Res$ is the residual system.
But the degree of $\Res$ is negative, indeed it is
$$d-\sum_{i=1}^{n+1} k(i)= d- \sum_{i=1}^{n+1} (m_1+\dots \widehat
{m_i}+\dots +m_{n+1}-(n-1)d)=$$
$$=d-(n\sum_{i=1}^{n+1} m_{i}-(n+1)(n-1)d)=n(nd-\sum_{i=1}^{n+1}m_i)<0,$$
therefore $D$ is not effective and this gives a contradiction.

If $s=n+2$, then since $nd<\sum_{i=1}^{n+2} m_{i}$ and $d\geq m_{n+2}$, we have 
$c=\sum_{i=1}^{n+1} m_{i}-(n-1)d>0$. 
So performing a Cremona transformation based at the first $n+1$ points, 
since we have $d-c=nd-\sum_{i=1}^{n+1} m_{i}<0$, we obtain a linear system with 
negative degree, so $D$ is again not effective, and we have a contradiction.

Notice, finally,  that if $s\leq n$ the inequality $\sum_{i=1}^s m_i\le nd$  is redundant. 
\end{proof}

We now improve Lemma \ref{base-locus-lemma}, giving results which 
predict the exact multiplicity with
which a cycle is contained in the base locus of the linear system.

The following proposition is more general, concerning also non linear divisors, anyway 
it applies in particular to the case of hyperplanes $L_{I(n-1)}$.

\begin{proposition}\label{-1divisors} 
Let $\LL_{n,d}(m_{1},...,m_{s})$ be a linear system and $D\in\LL$.
Let $W$ be the Weyl  group of the blow-up, $X_{(0)}^n$, at the $s$ base points of $\LL$ 
and let $F$ be a divisor in any Weyl orbit of $E_{i}$. 
If $\langle D, F\rangle\leq 0$, then $F$ is contained in the base locus of $\LL$ 
with multiplicity $-\langle D,F\rangle$.
\end{proposition}

\begin{proof}
Let $F=w^{-1}(E_{i})$ for $w\in W$. We note that
$\langle E_{i}, w(D)\rangle=\langle w(F),w(D)\rangle=\langle F, D \rangle<0$, hence
\cite[Lemma 4.3]{CDD} implies that any exceptional divisor $E_{i}$ is contained in the base locus of the divisor 
$w(D)$ with multiplicity equal to $-\langle E_{i} ,w(D)\rangle$. 
Applying the birational map of $\PP^{n}$ corresponding to the element
 $w^{-1}$ we obtain that $F$ is contained in the base locus of $D$ with the 
same multiplicity, that equals  $-\langle F, D\rangle$.
\end{proof}
In particular, the above proposition shows that, for $n=2$,  the multiplicity of containment 
of any $(-1)$-curve is given by the intersection product.
 
The following result is an easy consequence of Lemma \ref{base-locus-lemma} and concerns the case 
where the multiplicities of the first points equal the degree of the linear system.

\begin{proposition}
\label{bscone}
Let $\LL:=\LL_{n,d}(d^{t},m_{1},...,m_{s})$ be a non-empty linear system. 
Given $r \leq t-1$, let $L_{I(r)}$ be the $r$-cycle spanned by $r+1$ 
among the first $t$ points.
Then $\LL$ contains in its base locus the cycle $L_{I(r)}$ with multiplicity $d$.
\end{proposition}
\begin{proof}
Since the first $r+1$ points have multiplicity $d$ we obtain 
$k_{I(r)}= d$. Hence by Lemma \ref{base-locus-lemma} the cycle 
$L_{I(r)}$ is contained in the base locus with multiplicity at least $d$.
Obviously the multiplicity can not be higher than the degree, being $\LL$  not empty.
\end{proof}

Finally we give a result concerning the case of linear systems with at most $n+2$ points.
\begin{proposition}\label{sharpBLL}
Assume that $s\le n+2$ and let $\LL_{n,d}(m_{1},...,m_{s})$ be a non-empty linear system.
If $0\le r\le \min(n,s)-1$,  
then $\LL_{n,d}(m_{1},...,m_{s})$ contains in its base locus the cycle $L_{I(r)}$ 
with multiplicity $k_{I(r)}$, for any $I(r)\subseteq\{1,\dots,s\}$.
\end{proposition}

\begin{proof} 
For every multi-index $I(l)\subseteq\{1,\dots,s\}$ we write
$K_{I(l)}:=\sum_{i\in I(l)}m_i-ld$, so that
 $k_{I(l)}=\max(K_{I(l)},0)$.

Fixed a multi-index $I(r)\subseteq\{1,\dots,s\}$, 
we consider separately the following cases:
\begin{itemize}
\item[(1)] $k_{I(r)}=K_{I(r)}\ge 0$
\item[(2)] $k_{I(r)}\neq K_{I(r)}<0$.
\end{itemize}

\smallskip

Case (1). 
We denote by 
$$R:=\max\{l|K_{I(l)}\geq 0, I(r)\subseteq I(l)\},$$

If $r=n-1(=R)$ then our claim is true by Proposition \ref{-1divisors}.
Next, we consider separately the following cases:
\begin{itemize}
\item[(i)] $r< R$
\item[(ii)] $r=R$.
\end{itemize}

Case (i). {\it 
We assume that for any $I(R)$ such that $K_{I(R)}\ge0$
the cycle $L_{I(R)}$ is contained in $\LL$  with multiplicity $K_{I(R)}$, and
we prove that any cycle
$L_{I(r)}$ is contained in $\Bs(\LL)$ with multiplicity $K_{I(r)}$.}

Notice that $0\le K_{I(R)}\leq K_{I(r)}$, since $m_i\le d$.
Therefore all linear subcycles $L_{I(l)}$ of $L_{I(R)}$
are contained in $\Bs(\LL)$ with multiplicity at least 
$K_{I(l)}\geq 0$,  by Lemma \ref{base-locus-lemma}; in particular
$L_{I(r)}$ is contained at least $K_{I(r)}$ times.
Assume now by contradiction that $L_{I(r)}$ is contained in $\Bs(\LL)$ with 
multiplicity at least $1+K_{I(r)}$. 
We know that the linear cycle $L_{J(R-r-1)}$ is contained in the base locus 
with multiplicity at least 
$K_{J(R-r-1)}\ge 0$, where $J(R-r-1):=I(R)\setminus I(r)$. 
For any point $p$ in the cycle $L_{I(r)}$ and $p'$ in the cycle $L_{J(R-r-1)}$ 
we get that the line spanned by $p$ and $p'$ is contained in the base locus 
with multiplicity at least $1+K_{I(r)}+K_{J(R-r-1)}-d=1+K_{I(R)}$ 
and this is a contradiction.
Hence $L_{I(r)}$ is contained in $\Bs(\LL)$ with multiplicity $K_{I(r)}$.

Case (ii). {\it 
 We prove that
for any $I(R)$ such that $K_{I(R)}\ge0$
the cycle $L_{I(R)}$ is contained in $\LL$  with multiplicity $K_{I(R)}$.}

We prove this claim for any non-empty linear system in $\PP^n$
by using backward induction on $R$.
Given $R\le n-2$, assume now that for every non-empty 
linear system $\LL$ in $\PP^n$ such that
$$\max\{l|K_{I(l)}\geq 0, I(r)\subseteq I(l)\}=R+1,$$
for any multi-index $I(R+1)$ such that $K_{I(R+1)}\ge0$,
we know that the cycle $L_{I(R+1)}$ is contained in the base locus with multiplicity $K_{I(R+1)}$.
We prove the statement for a non-empty linear system with 
$$\max\{l|K_{I(l)}\geq 0, I(r)\subseteq I(l)\}=R.$$
Let $I(R)=\{i_1,\ldots, i_{R+1}\}$ 
such that $K_{I(R)}\ge 0$.

We first consider the case $s\le n$.
By the effectivity of $\LL$ we have $nd\geq m_{1}+\dots+m_{s}$. 
If $m_i=d$ for all $i\in I(r)$ we conclude by Proposition \ref{bscone}, hence we can assume that 
$m_i<d$ for some $i\in I(r)$. 
This implies that $0< d-K_{I(R)}\leq d$. 
We consider now the subsystem $\LL'$ of $\LL$ 
obtained by adding another general point $p_{s+1}$ of multiplicity $d-K_{I(R)}$.
This is a linear system based at, at least, $R+2$ points,  
and which is clearly effective, by Lemma \ref{effectivity}.
Now, for  the linear system $\LL'$, we have $K_{I(R+1)}=0$ by construction, for all $I(R+1)\ni s+1$,
and therefore by induction $L_{I(R+1)}$ is contained in $\Bs(\LL')$ with multiplicity $K_{I(R+1)}$.
Hence by applying case (i) we deduce that
the cycle $L_{I(R)}$ is contained in $\Bs(\LL')$ with multiplicity $K_{I(R)}$.
We finally note that
 $\Bs(\LL')\supseteq\Bs(\LL)$,  hence we deduce that 
$L_{I(R)}$ is also contained in $\Bs(\LL)$ with multiplicity exactly $K_{I(R)}$.

Now we consider the case $n+1\le s\le n+2$. We define
\begin{align*}
q_\LL:&=\min\{(R+1)d-(m_{i_1}+\dots+m_{i_{R+1}}+m_{j})| j\not\in I(R)\}\\
&=\min\{d-m_{j}-K_{I(R)}|j\not\in I(R)\}
=\min\{-K_{I(R+1)}|I(R+1)=I(R)\cup\{j\}\}.
\end{align*}
Let $I(R+1)=\{i_1,\ldots,i_{R+1},i_{R+2}\}$ 
for some $i_{R+2}\not\in I(R)$ such that $-K_{I(R+1)}=q_\LL$. 
Note that $q_\LL>0$, by definition of $R$, 
and recall that $K_{I(R)}\geq 0$, hence 
we have $d>m_{i_{R+2}}$.
Now if $m_i=d$ for all $i\in I(R)$, then we apply Proposition \ref{bscone} 
to conclude that $L_{I(R)}$ is contained in $\Bs(\LL)$ 
with multiplicity $d$. 
Otherwise, after possibly reordering the points, we can assume also 
that $m_{i_{R+1}}<d$. 
Now
we restrict $\LL$ to a hyperplane passing through $n$ 
points, such that $p_{i_1},\dots, p_{i_R}$  are among them and $p_{i_{R+1}}, p_{i_{R+2}}$ are not.
The residual linear system $\Res$ is non-empty by construction, 
in fact it satisfies conditions (\ref{effectivity}), because $\LL$ does. 
It is easy to check that $q_{\Res}=q_\LL-1$. 
We repeat this procedure, after possibly 
reordering the points, 
until either all points on $L_{I(R)}$ 
have multiplicity equal to the degree, or we get a residual, 
that we denote again by $\Res$ abusing notation, for which $q_{\Res}=0$. 
Since, the linear system $\LL$ contains the subsystem given by the hypersurfaces
reducible to the sum of hypersurfaces of $\Res$ and hyperplanes not containing
the cycle $L_{I(R)}$, 
then clearly we have $\Bs(\LL)\subseteq\Bs(\Res)$.
Now note that $q_{\Res}=0$, in particular in $\Res$ there is a cycle
$L_{I(R+1)}$ with $K_{I(R+1)}=0$. Hence we apply step (i) to conclude
that $L_{I(R)}$ is contained in $\Bs(\Res)$ with multiplicity $K_{I(R)}$. Hence it follows
that $L_{I(R)}$ is contained in $\Bs(\LL)$ with multiplicity $K_{I(R)}$. 

\smallskip

Case (2). 
An easy remark is that $L_{I(r)}$ contains at least two points, 
say $p_{i_r}, p_{i_{r+1}}$,
with $d>m_{i_r},m_{i_{r+1}}$, since $d\ge 1$. 
We consider the restriction of $\LL$ 
to a hyperplane passing through $\min(n,s)$ points of $\LL$, 
such that $p_{i_1},\dots,p_{i_{r-1}}\in L_{I(r)}$ are among them and $p_{r},p_{r+1}$ are not.
In this way, as above, the residual system $\Res$ 
satisfies conditions (\ref{effectivity}) hence it is non-empty.
We can repeat this procedure until we get a 
residual for which the corresponding number $K_{I(r)}$ is null. 
Therefore, by the previous case, $L_{I(r)}$ 
is not contained in the $\Bs(\Res)$, so
it is  not contained  in $\Bs(\LL)$.
\end{proof}

\section{Linear expected dimension and linear speciality}
\label{cones-section}

In this section we will introduce the main notion of the paper, 
which is the notion of linear speciality.

We consider first, as an illustrative example, 
the case of cones 
which is  easy and well-understood.
Note that any divisor in the linear system $\LL_{n,d}(d^s)$
is a cone with vertex the linear space spanned by the $s$ points, 
for any $0\le s\leq n$ (this immediately follows 
from Proposition \ref{bscone}).
On the other hand if $s\ge n+1$ the linear system $\LL_{n,d}(d^s)$ is empty.
We get then the following result. 

\begin{proposition}\label{cones-theorem}
Given $0\le s\leq n+1$ and the linear system $\LL=\LL_{n,d}(d^s)$, we have
\begin{equation}\label{h1-cones}
\hh^1(\PP^n,\LL)=\sum_{i=2}^s (-1)^i\binom{n+d-i}{n}\binom{s}{i}.
\end{equation}
\end{proposition}

\begin{proof}
Let $Z$ be a collection of $s$ multiplicity-$d$ points in $\PP^n$. We have:
$$\hh^0(\PP^n,\II_Z\otimes \OO(d))= \hh^0(\PP^{n-s},\OO(d))= \binom{n-s+d}{d}$$
if $s\le n$, and
$\hh^0(\PP^n,\II_Z\otimes \OO(d))= 0$ if $s=n+1$.
Hence
$$\hh^1(\PP^n,\II_Z\otimes \OO(d))=\binom{n-s+d}{d}- \binom{n+d}{d}+s\binom{n+d-1}{n}.$$

It is easy to prove that
\begin{equation}\label{formula}
\binom{n-s+d}{d}- \binom{n+d}{d}+s\binom{n+d-1}{n}=
\sum_{i=2}^s(-1)^i\binom{n+d-i}{n}\binom{s}{i}\end{equation}
by double induction on $n\ge 1$ and $d\ge 1$.
This concludes the proof.
\end{proof}

We can interpret formula \eqref{h1-cones} in the following way:
there are $\binom{s}{i}$ 
linear $(i-1)$-cycles contained in the base locus, each of them
having multiplicity of containment $d$ and giving an obstruction equal to
of $(-1)^i\binom{n+d-i}{n}$ to the speciality.

This example suggests that, in general, 
if a linear subspace $\PP^r$ is contained in the 
base locus of any linear system $\LL$ with multiplicity $k$, then its 
contribution to the speciality of $\LL$ is
$$(-1)^{r+1}\binom{n+k-r-1}{n}.$$

Therefore we give the following new definition of virtual (and expected) dimension which takes in account these linear obstructions.

\begin{definition}\label{new-definition}
Given a linear system $\LL=\LL_{n,d}(m_1,\ldots,m_s)$, for any integer 
$-1\le r\le  s-1$ and
for any multi-index $I(r)=\{i_1,\ldots,i_{r+1}\}\subseteq\{1,\ldots,s\}$, 
with the convention $I(-1):=\emptyset$,
we adopt the following  notation, which extends formula (\ref{definition-of-k}):
$$
k_{I(r)}:=\max(m_{i_1}+\ldots+m_{i_{r+1}}-rd,0).
$$

The {\em linear virtual dimension} of $\LL$ is the number
\begin{equation}\label{linvirtdim}
\sum_{r=-1}^{s-1}\sum_{I(r)\subseteq \{1,\ldots,s\}} (-1)^{r+1}\binom{n+k_{I(r)}-r-1}{n} -1.
\end{equation}

The {\em linear  expected dimension} of $\LL$, denoted by $\ldim(\LL)$ is 
defined as follows: if
the linear system $\LL$ is contained in a linear system whose linear virtual dimension
 is negative, then
we set $\ldim(\LL)=-1$, otherwise we define $\ldim(\LL)$ to be
 the maximum between the linear virtual dimension of $\LL$ and $-1$.

We  say that $\LL$ is {\em linearly special} if $\dim(\LL)\neq\ldim(\LL)$.
Otherwise we say that $\LL$ is {\em linearly non-special}. 
\end{definition}

Notice that, according to this definition, we argue in Proposition \ref{cones-theorem}  that
any linear system of the form $\LL_{n,d}(d^s)$ is linearly non-special.

\begin{remark}
Note that if $\LL$ is non-empty, then the terms corresponding to $n\le r \le s-1$
in the sum of formula \eqref{linvirtdim} vanish. In the next section we will prove that,
in this case,
$$
\ldim(\ls)=\chi(\tilde{\LL})-1,
$$
where 
$\chi(\tilde{\LL})$ is the Euler  characteristic of 
the sheaf associated to the strict transform of $\LL$
 after blowing up the linear base locus. 
\end{remark}

\begin{remark}\label{old5.1}
It is easy to prove that
$\vdim(\LL)\le \ldim(\LL)$ for any linear system $\LL$.
On the other hand, we expect that $\ldim(\LL)\le \dim(\LL)$. 
This inequality is not obvious and it  
is indeed equivalent to the {\it weak Fr\"oberg-Iarrobino conjecture}, see \cite[Conjecture 4.5]{Chandler}.
Note also that our definition of linear virtual dimension is equivalent to 
\cite[Definition 6]{Chandler}.
\end{remark}

Also the family of homogeneous linear systems with only triple points 
 $\LL_{n,d}(3^s)$, for arbitrary  $s$, is well suited to test our approach. 
By means of computer-aided computations \cite{macaulay}, 
we calculated the dimensions of these linear systems for $n=3,4,5$ and low degree, getting the following (partial)
classification:

In $\PP^3$, if $d\le38$, the only special linear systems are:
\begin{itemize}
\item $\LL_{3,3}(3^2), \LL_{3,3}(3^3)$ \quad  linearly non-special,
\item $\LL_{3,4}(3^2), \LL_{3,4}(3^3), \LL_{3,4}(3^4)$ \quad linearly non-special,
\item $\LL_{3,6}(3^9)$ \quad  linearly special.
\end{itemize}

In $\PP^4$, if $d\le 10$, the only special linear systems are:
\begin{itemize}
\item $\LL_{4,3}(3^2), \LL_{4,3}(3^3), \LL_{4,3}(3^4)$ \quad linearly non-special,
\item $\LL_{4,4}(3^2), \LL_{4,4}(3^3), \LL_{4,4}(3^4), \LL_{4,4}(3^5)$  linearly non-special,
\item $\LL_{4,6}(3^{14})$ \quad  linearly special.
\end{itemize}

In $\PP^5$, if $d\le 7$, the only special linear systems are:
\begin{itemize}
\item $\LL_{5,3}(3^2),\ldots, \LL_{5,3}(3^5)$ \quad linearly non-special,
\item $\LL_{5,4}(3^2), \ldots, \LL_{5,4}(3^6)$ \quad linearly non-special.
\end{itemize}

In other words our experiments show that 
there are no linearly special linear systems with triple points in $\PP^3$ with degree
$\le38$ except for the case $\LL_{3,6}(3^9)$,
there are no linearly special linear systems with triple points in $\PP^4$ with degree
$\le10$ except for $\LL_{4,6}(3^{14})$,
and there are no linearly special linear systems with triple points in $\PP^5$ with degree
$\le7.$

The first line in the three groups of examples corresponds precisely to the case of cones, see Proposition \ref{cones-theorem}.

Notice that the two exceptional cases $\LL_{3,6}(3^9)$ and $\LL_{4,6}(3^{14})$ can be explained by the 
fact that there is a quadric in the base locus which gives speciality. 
They were predicted both the by Fr\"oberg-Iarrobino conjecture and by the Laface-Ugaglia conjecture, see Section \ref{final-section} for
a more detailed explanation about these conjectures.

Besides this two special cases, all other cases have $s\le n+2$. 
In  Corollary \ref{corollary-monster},
 we will give a precise description and an explicit computation of their speciality.

\section{Linear systems with at most $n+2$ points}\label{up to $n+2$}
\label{toric-case-section}

\subsection{Blowing up: construction and notation}
\label{section4.1}

Let $p_1,\dots,p_s$ be general points in $\PP^n$. 
For every integer $0\leq r\le \min(n,s)-1$ we denote by $I(r)\subseteq \{1,\dots,s\}$ 
a multi-index of length $|I|=r+1$,
and by $L_{I(r)}$ the unique $r$-cycle through the points 
$\{p_i,\ i\in I(r)\}$: $L_{I(r)}\cong\PP^r\subseteq\PP^n$.
Notice that $L_{I(0)}=p_i$ is a point.

Let $\mathcal{I}$ be a set of subsets of $\{1,\dots,s\}$ such that
\begin{enumerate}
\item $\{i\}\in\mathcal{I}$, for all $i\in\{1,\dots,s\}$;
\item if $I\subset J$ and $J\in\mathcal{I}$, then $I\in\mathcal{I}$.
\end{enumerate}
Let $\Lambda=\Lambda(\mathcal{I})\subset\PP^n$ be
the subspace arrangement corresponding to $\mathcal{I}$, i.e.  the (finite) union of the
linear cycles $L_I$ for $I\in\mathcal{I}$.
Let $\bar{r}$ be the dimension of the biggest linear cycle in $\Lambda$, 
i.e.\ $\bar{r}=\max_{I\in\mathcal{I}}(|I|)-1$. 
Write $\Lambda=\Lambda_{(1)}+\cdots+\Lambda_{(\bar{r})}$, 
where $\Lambda_{(r)}=\cup_{I(r)\in\mathcal{I}} L_{I(r)}$.

Assume moreover that $\mathcal{I}$ satisfies the following condition 
\begin{itemize}\label{disjoint complementary cycles}
\item[(3)] if $I,J\in\mathcal{I}$, then $L_{I}\cap L_J=L_{I\cap J}$.
\end{itemize}
Notice that this condition is obviously satisfied when $s\le n+1$.

As in Section \ref{base-locus-section}, we denote by 
$\pi_{(0)}^n:X_{(0)}^n\to\PP^n$
the blow-up of $\PP^n$
at $p_1,\dots,p_s$, with $E_1,\dots,E_s$ exceptional divisors. 
Let us also consider the following sequence of blow-ups:
\[
X_{(\bar{r})}^n\stackrel{\pi_{(\bar{r})}^n}{\longrightarrow} 
\cdots \stackrel{\pi_{(3)}^n}{\longrightarrow}
X_{(2)}^n\stackrel{\pi_{(2)}^n}{\longrightarrow}X_{(1)}^n
\stackrel{\pi_{(1)}^n}{\longrightarrow}X_{(0)}^n,
\]
where $X_{(r)}^n\stackrel{\pi_{(r)}^n}{\longrightarrow}X_{(r-1)}^n$
denotes the blow-up of $X_{(r-1)}^n$ along 
the strict transform of $\Lambda_{(r)}\subset\PP^n$,
via $\pi_{(r-1)}^n\circ\cdots\circ\pi_{(0)}^n$.  
Let us denote by
$E_{I(r)}$ the exceptional divisors of the cycles $L_{I(r)}$, for any $I(r)\in\mathcal{I}$.
We will denote, abusing notation, by $H$ the pull-back in $X_{(r)}^n$
of $\mathcal{O}_{\PP^n}(1)$ and by
${E}_{I(\rho)}$, for $0\le \rho \le r-1$, the pull-backs in $X_{(r)}^n$
of the exceptional divisors of $X^n_{(\rho)}$, respectively.

\begin{remark}\label{blow up of hypersurface}
Notice that, in the case $\bar{r}=n-1$,  the map
$X_{(n-1)}^n\stackrel{\pi_{(n-1)}^n}{\longrightarrow} X_{(n-2)}^n$ is an isomorphism and 
in particular $\Pic(X_{(n-1)}^n)\cong (\pi^n_{(n-1)})^\ast \Pic(X_{(n-2)}^n)$. 
Thus, in our notation, for every $I(n-1)\subseteq \{1,\dots,s\}$ we have
 $$E_{I(n-1)}=H-\sum_{\substack{I(\rho)\subseteq I(n-1),\\0\le \rho\le n-2}}E_{I(\rho)}.$$
\end{remark}

\subsubsection*{Intersection theory on the blow-up $X_{(r)}^n$}

The Picard group of $X^n_{(r)}$ is 
$$\textrm{Pic}(X^n_{(r)})=\langle H, E_{I(\rho)}: 0\le \rho \le r-1  \rangle.$$

\begin{remark}\label{blow up preserves h^i}
For $r=1,\dots, n-1$, if $F$ is any divisor on $X^n_{(r-1)}$, then for any $i\geq 0$, we have
\[
h^i(X^n_{(r)}, (\pi^n_{(r)})^\ast F)=h^i(X^n_{(r-1)},F).
\]
It follows from Zariski connectedness Theorem and by the projection formula (see for instance \cite{hartshorne}
or \cite[Lemma 1.3]{laface-ugaglia-BullBelg} for a more detailed proof.).
\end{remark}

Let $\pi=\pi_{(r)}^n\circ\cdots\circ\pi_{(0)}^n$. 
Given $0\leq \rho\le \min(n,s)-1$ and any multi-index $I(\rho)=\{i_1,\ldots,i_{\rho+1}\}$, 
we denote by $H_{I(\rho)}$ the strict transform via $\pi$ 
of a hyperplane $\mathcal{H}$ of $\PP^n$ 
containing the points $p_{i_1},\ldots,p_{i_{\rho+1}}$.

Note that the total transform  of $\mathcal{H}$ is 
$$\pi^*(\mathcal{H})=H_{I(\rho)}+\sum_{\substack{J\subseteq I(\rho),\\ |\rho|\le r+1}} E_{J},$$ 
since the hyperplane $\mathcal{H}$ contains the cycle $L_J$, for any $J\subseteq I(\rho)$, and the cycle has been blown up if its length is at most $r+1$.

Notice that $H_{I(\rho)}$ is the blow-up of the hyperplane $\mathcal{H}$ at the points 
$p_{i_1},\ldots,p_{i_{\rho+1}}$ and at all the cycles $L_J$ for any $J\subseteq I(\rho)$ of length at most $r+1$. 
Denoting by $h$ the pull-back of $\mathcal{O}_{\mathcal{H}}(1)$ and 
by $e_J$ the corresponding exceptional divisors, we have
$$\textrm{Pic}(H_{I(\rho)})=\langle h, e_J: J\subset I(\rho), |J|\le \min(n-2,r+1) \rangle.$$
Then we have:
$$h={{H}}_{|H_{I(\rho)}},$$
and, for any multi-index $J$,
\begin{align*}
{{E_J}}_{|H_{I(\rho)}}&=0, &\textrm{ if } J\cap {I(\rho)}= \emptyset, |J|\le n-2 ,\\
{{E_J}}_{|H_{I(\rho)}}&=e_{J\cap {I(\rho)}}, &\textrm{ if } J\cap {I(\rho)}\neq \emptyset, |J|\le n-2.
\end{align*}
If $r=n-2$, we also have some exceptional divisor $E_J$ with $|J|=n-1$, 
and in this case, if $J\subseteq I(\rho)$ we have
$${{E_J}}_{|H_{I(\rho)}}= h-\sum_{K\subsetneq J}e_K.$$

\subsubsection*{The geometry of the exceptional divisors}

Fix $0\leq r\le \min(n,s)-1$ and 
consider an exceptional divisor $E_{I(r)}$ in $X^n_{(r)}$, 
for $I(r)=\{i_1,\dots,i_{r+1}\}\in\mathcal{I}$.
Notice that 
$$E_{I(r)}\cong X^r_{(r-1)}\times \PP^{n-r-1} \subset X^n_{(r)},$$
where $X^r_{(r-1)}$ denotes the blow-up of $L_{I(r)}\cong\PP^r$ along all
linear $\rho$-cycles, $0\le \rho\le r-1$, spanned by the points $p_{i_1},\dots,p_{i_{r+1}}$.

Let us denote 
$$\textrm{Pic}(X^r_{(r-1)})=\langle h, e_{I(\rho)}: I(\rho)\subset I(r), \rho\le r-2 \rangle.$$

Recall that the canonical sheaf of $X^r_{(r-1)}$ is
\begin{equation}\label{canonical}
\OO_{X^r_{(r-1)}}(-(r+1)h+(r-1)\sum e_i+(r-2)\sum e_{I(1)}+
\ldots +\sum e_{I(r-2)})
\end{equation}

For any multi-index $I(r)$ we will denote by $x_{I(r)}$ the following divisor on $X^r_{(r-1)}$
\begin{equation}\label{def-x}
x_{I(r)}= rh -\left((r-1)\sum e_i+(r-2)\sum e_{I(1)}+\ldots+\sum e_{I(r-2)}\right)
\end{equation}

Note that $x_{I(r)}$ is the Cremona transform of the hyperplane of $L_{I(r)}$. 

\begin{lemma}\label{lemma-EdotE} 
For any exceptional divisor $E_{I(r)}$ in $X^n_{(r)}$, we have
\begin{equation}\label{EdotE}
{E_{I(r)}}_{|E_{I(r)}} \cong \OO_{X^r_{(r-1)}\times \PP^{n-r-1}}(-x_{I(r)}, -1),
\end{equation}
where $x_{I(r)}$ is defined in \eqref{def-x}.
\end{lemma}

\begin{proof}
Let  $I=I(r)\in\mathcal{I}$ be the set of indices parametrizing the $r+1$ fundamental points of the linear cycle $L_{I(r)}$
in $X^n_{(0)}$ whose corresponding exceptional divisor in $X^n_{(r)}$ is $E_{I(r)}$.
For  $1\le \rho\le r-1$, set $\varepsilon_{(\rho)}:=\sum_{I(\rho)\subset I} e_{I(\rho)}$
 to be the exceptional divisor
 of  all linear cycles of dimension $\rho$ 
of $\PP^n$ that are contained in $L_{I(r)}$,
namely of ${\Lambda_{(\rho)}}_{|L_{I(r)}}\subset L_{I(r)}\cong\PP^r$.
 Notice that, being $\pi^r_{(r-1)}:X^r_{(r-1)}\to X^r_{(r-2)}$ an isomorphism, then $\varepsilon_{(r-1)}$ is the 
 pull-back of the sum of 
fundamental hyperplanes $H_{I(r-1)}$ of $X^r_{(0)}$: 
\begin{align*}
\varepsilon_{(r-1)}&=\sum_{I(r-1)\subset I} (h-\sum_{i\in I(r-1)} e_i-\sum_{I(1)\subset I(r-1)}e_{I(1)}-\cdots-
\sum_{I(r-2)\subset I(r-1)}e_{I(r-2)})\\
\ &= (r+1)h-r \sum_{i\in I}e_i-(r-1)\sum_{I(1)\subset I}e_{I(1)}-\cdots- 2\sum_{I(r-2)\subset I}e_{I(r-2)}.
\end{align*}
Set now
 $\psi^\ast: = (\pi^r_{(r-1)})^\ast \cdot (\pi^r_{(r-2)})^\ast \cdots\cdot (\pi^r_{(1)})^\ast $. 
By using \cite[B.6.10]{fulton}, we compute the normal bundle $N_{X^r_{(r-1)}| X^n_{(r)}}$ of 
$X^r_{(r-1)}$ in $X^n_{(r)}$ and we get:
\begin{align*}
{N_{X^r_{(r-1))}|X^n_{(r-1)}} }&=
\psi^\ast (N_{X^r_{(0)}|X^n_{(0)}})\otimes-\left(\sum_{\rho=1}^{r-2} \varepsilon_{(\rho)}+\varepsilon_{(r-1)}\right)\\
&= \OO_{X^r_{(r-1)}}\left(\left(h-\sum_{i\in I}e_i\right) -\sum_{\rho=1}^{r-2} \varepsilon_{(\rho)}-\varepsilon_{(r-1)})\right)^{\oplus n-r}\\
& ={\OO_{X^r_{(r-1)}}(-x_{I(r)})}^{\oplus n-r}
\end{align*}
hence we deduce \eqref{EdotE}.
\end{proof}

\begin{remark}\label{remark-intersection}
Given a multi-index $I$ and $F$ any divisor in $X^r_{(r-1)}$,
if $F$ contains $k$ times the cycle $L_I$ in its base locus, then we have
$${F}_{|E_I}\cong \OO_{X^r_{(r-1)}\times \PP^{n-r-1}}(-k x_{I(r)},0)$$ 
where $x_{I(r)}$ is as in \eqref{def-x}.
\end{remark}

\subsection{The main  theorem}
In this section we give our main result, Theorem \ref{monster-theorem}, 
concerning the cohomology of all the strict transforms of a linear system. 
Note that
as a consequence of this technical result 
we obtain that a non-empty linear system $\LL_{n,d}(m_1,\dots,m_s)$ 
with $s\le n+2$ is always linearly non-special (see Corollary \ref{corollary-monster}).

Let $\LL=\LL_{n,d}(m_1,\dots,m_s)$  be a non-empty linear system on $\PP^n$. 
Elements $D$ of $\LL$ are in bijection with  divisors on $X_{(0)}^n$ of the 
following form:
\[
 D_{(0)}:= dH-\sum_{i=1}^s m_i E_i.
\]
In the blow-up $X^n_{(r)}$ of $X^n_{(r-1)}$  along 
the union of the strict transforms of all $r$-cycles $L_{I(r)}$,
the total transform of $D_{(r-1)}\subset X^n_{(r-1)}$ is 
\begin{equation}\label{complete transform}
(\pi^n_{(r)})^\ast D_{(r-1)}=dH-\sum_{\substack{I(\rho), \\ 0\le \rho\le r-1}} k_{I(\rho)} E_{I(\rho)},
\end{equation}
while the strict transform of $D_{(r-1)}$ is
\begin{equation}\label{strict transform}
D_{(r)}= dH-\sum_{\substack{I(\rho),\\ 0\le \rho\le r-1}} k_{I(\rho)} E_{I(\rho)}-\sum_{I(r)} k_{I(r)} E_{I(r)}=
dH-  \sum_{\substack{I(\rho), \\ 0\le \rho\le r}} k_{I(\rho)} E_{I(\rho)},
\end{equation}

For the sake of simplicity throughout this paper we will abbreviate by $H^i(X^n_{(r)},D_{(r)})$, or by $H^i(D_{(r)})$, the
cohomology group $H^i(X^n_{(r)},\mathcal{O}_{X^n_{(r)}}(D_{(r)}))$.

\begin{remark}\label{hyperplanes in base locus}
If the linear base locus of $\ls$ has maximal dimension $n-1$, then
there exists an effective divisor $\Delta$ in $X^n_{(0)}$ such that
\[
H^i(X^n_{(n-1)},D_{(n-1)})\cong H^i(X^n_{(n-2)},\Delta_{(n-2)}),
\]
where $\Delta_{(n-2)}$ is the strict transform of $\Delta$ in $X^{n}_{(n-2)}$.
\end{remark}
\begin{proof} 
Recall that a hyperplane $L_{I(n-1)}$ through $n$ points of multiplicity respectively 
$m_{i_1},\dots, m_{i_n}$
is contained in the base locus of $D=dH-\sum_{i=1}^s m_i E_i$ if and only if 
$k_{I(n-1)}=\sum_{j=1}^n m_{i_j}-(n-1)d\ge 1$,
 by Proposition \ref{-1divisors}.
All hyperplanes for which $k_{I(n-1)}\ge1$ split off $D$ 
$k_{I(n-1)}$-many times, and the residual part is a divisor 
$\Delta=\delta H-\sum_{i=1}^{s}\mu_i E_i$
 on $X^n_{(0)}$, with $\delta=d-\sum_{I(n-1)}k_{I(n-1)} \ge 0$ and 
 $\mu_i=m_i-\sum_{I(n-1)\ni i} 
k_{I(n-1)}\ge 0$, for  $i=1,\dots, s$.

Clearly $\Delta$ is effective 
and its linear base locus has maximal dimension $\le n-2$.
Denote by $\Delta_{(n-2)}$ the strict transform of $\Delta$ in $X^n_{(n-2)}$. 
The conclusion easily  follows as the strict transform $D_{(n-1)}$ equals
the total transform in $X^n_{(n-1)}$ of $\Delta_{(n-2)}$ in $X^n_{(n-2)}$.
\end{proof}

The following theorem is the main result of this section:

\begin{theorem}\label{monster-theorem}
Given a non-empty linear system $\LL=\LL_{n,d}(m_1,\dots,m_s)$ and $D\in\LL$,
then the following holds.
\begin{enumerate}\item[(i)]
Assume that $r$ is an integer such that $1\le r \le\min(n,s)-1$. Then
$h^i(X^n_{(r)},D_{(r)})=h^i(X^n_{(r-1)},D_{(r-1)})$, for $i\leq r-1$, and 
$h^i(X^n_{(r)},D_{(r)})=0$, for $i\ge r+2$. Moreover
\[
h^r(X^n_{(r)},D_{(r)})-h^{r+1}(X^n_{(r)},D_{(r)})=
h^r(X^n_{(r-1)},D_{(r-1)}) -
\sum_{I(r)} {{n+k_{I(r)}-r-1}\choose{n}}.
\] 
\item[(ii)] 
Assume that $s\le n+2$ and that 
 $r$ is such that $0\le r\le\min(n,s)-1$. Then
$h^i(X^n_{(r)},D_{(r)})=0$,  for $i\ge 1$, $i\neq r+1$. Moreover, if $\bar{r}$
is the dimension of the linear base locus in $\ls$, then 
$h^{r+1}(X^n_{(r)},D_{(r)})=0$, for $r\ge \bar{r}$.
\end{enumerate}
\end{theorem}

\begin{remark}
Notice that Theorem \ref{monster-theorem}, part $(ii)$ states
 that for a linear system $\LL_{n,d}(m_1,\dots,m_s)$, with $s\le n+2$, 
and any divisor $D\in\LL$ we have
\[ 
h^i(X^n_{(\bar{r})},D_{(\bar{r})})=0,\quad \mbox{ for all $i\ge1$.}
\]
This means that
after blowing up the linear base locus we get non-speciality of the strict transform
 $D_{(\bar{r})}$. 
In other words if $s\le n+2$ we are able to affirmatively answer to Question \ref{question}. 
\end{remark}

The following is an immediate consequence of 
Theorem \ref{monster-theorem}.

\begin{corollary}\label{corollary-monster}
If $\LL=\LL_{n,d}(m_1,\dots,m_s)$ is a non-empty linear system 
with $s\le n+2$, 
then its speciality is given  by
\begin{equation}
h^1(\ls)=\sum_{\substack{I(r),\\ 1\leq  r\leq \bar{r}}}
 (-1)^{r-1}\binom{n+k_{I(r)}-r-1}{n},
\end{equation}\label{total h^1}
where $\bar r$ is the dimension of the base locus. 

In particular $\dim(\LL)=\ldim(\LL)$ and $\LL$ is linearly non-special.
\end{corollary}

More precisely the following corollary describes the cohomology of the strict transform $D_{(r)}$
for any $r$.

\begin{corollary}
Let $\LL=\LL_{n,d}(m_1,\dots,m_s)$ be a non-empty linear system 
with $s\le n+2$ and with linear base locus of dimension $\bar{r}$.
For any divisor $D\in\LL$ and for  $0\le r \le \bar{r}$ we have
$$h^i(D_{(r)})=0,\quad i\ne 0, r+1,$$ 
$$
h^{r+1}(D_{(r)})=\sum_{\substack{I(\rho),\\ r+1\leq  \rho\leq \bar{r}}}
 (-1)^{\rho-r-1}\binom{n+k_{I(\rho)}-\rho-1}{n}.
$$
\end{corollary}

Note that, by Lemma \ref{effectivity}, 
the previous corollaries hold
for all linear systems satisfying conditions \eqref{condition-mult}.

\subsection{Proof of the main theorem}

We will split the proof of Theorem \ref{monster-theorem} in various steps and 
for this purpose we need to introduce the following notation.

Given the integers $n\ge 1$, $s\ge 0$ and $r$ with $1\le r\le\min(n,s)-1$,
we abbreviate by $(A_{n,s,r})$, $(B_{n,s,r})$ the following statements.\\
\begin{quote}
\begin{enumerate}\item[$(A_{n,s,r}):$]
 For any linear system $\LL=\LL_{n,d}(m_1,\dots,m_s)$ and  $D\in\LL$   
we have:
\begin{align*}
\quad \quad \quad h^i(D_{(r)})&=h^i(D_{(r-1)}),\quad  \mbox{ for $i\leq r-1$},\\
h^i(D_{(r)})&=0,\quad  \mbox{for $i\ge r+2$},\\
h^r(D_{(r)})&-h^{r+1}(D_{(r)}) =h^r(D_{(r-1)})-
\sum_{I(r)} {{n+k_{I(r)}-r-1}\choose{n}}.
\end{align*}\end{enumerate}

\begin{enumerate}\item[$(B_{n,s,r}):$]
For any linear system $\LL=\LL_{n,d}(m_1,\dots,m_s)$ and $D\in \LL$, 
 for any integer $l_{I(r)}$ with  $0\le l_{I(r)}\le\min(r,k_{I(r)})$ we have:
\[
h^i(D_{(r)})=h^i(D_{(r)}+\sum_{I(r)} l_{I(r)} E_{I(r)}), \quad 
\mbox{ for $i\ge 0$.}
\] 
\end{enumerate}\end{quote}

Given the integers $n,s,R$, with $n\ge 1$, $0\le s\le n+2$ and $0 \le R\le\min(n,s)-1$,
we abbreviate by $(C_{n,s,R})$ the following statement.\\
\begin{quote}
\begin{enumerate}\item[$(C_{n,s,R}):$]
For any linear system $\LL=\LL_{n,d}(m_1,\dots,m_s)$ 
with linear base locus of dimension $R$,  any $D\in\LL$ and
 any $r$ with $0\le r \le\min(n,s)-1$ we have:
\begin{align*}
h&^i(D_{(r)})=0,\quad  \mbox{ for $i\ge 1$, $i\neq r+1$},\\
h&^{r+1}(D_{(r)})=0, \quad \mbox{ for $r\ge R$}.\\
\end{align*}
\end{enumerate}\end{quote}

The proof of Theorem \ref{monster-theorem} will be by induction
and the two following propositions provide the inductive steps.

\begin{proposition}\label{toric implies normal bundle}
Let $n\ge 2$, $s\ge0$ and $1\le r\le\min(n,s)-1$. 
If for any $\rho$ such that $1\le \rho\le r-1$
statements 
$(A_{r,r+1,\rho})$, $(B_{r,r+1,\rho})$ and statement  $(C_{r,r+1,r-1})$ hold, 
 then statements $(A_{n,s,r})$ and $(B_{n,s,r})$ hold.
\end{proposition}

\begin{remark} 
Notice that, by applying Proposition \ref{toric implies normal bundle} in the case $r=1$, 
we get that statement 
$(C_{1,2,0})$, which  holds trivially true,
implies $(A_{n,s,1})$ and $(B_{n,s,1})$, for any $n$ and $s$.
This result was already proved by Laface and Ugaglia,
see \cite[Theorem 2.1]{laface-ugaglia-BullBelg}. 
\end{remark}

\begin{proof}[Proof of Proposition \ref{toric implies normal bundle}]
We define $\mathcal{I}:=\{I\subset\{1,\dots,s\}: k_I>0\}$, so that
the set of linear cycles $\Lambda(\mathcal{I})$ is the support of the linear 
base locus of $\LL$. 
Notice that, under the effectivity assumption on $\ls$, the set $\Lambda(\mathcal{I})$ is a subspace arrangement satisfying conditions
(1),(2) and (3) of Section \ref{section4.1}. 
In particular (2) is satisfied since $m_i\le d$, and (3) holds because it is easy to prove, 
by using \eqref{condition-mult}, that two disjoint index sets $I$ and $J$ belong to
 $\mathcal{I}$ only if $|I|+|J|\le n+1$.

Fix   a multi-index $I=I(r)\in\mathcal{I}$.
Let $E_I$ be the exceptional divisor of the corresponding linear $r$-cycle $L_I\subset\PP^n$ and denote by $F_I$  
the following divisor on $X_{(r)}^n$:
\[
F_I=(\pi^n_{(r)})^\ast D_{(r-1)}-\sum_{J\prec I} k_J E_J,
\] 
where $\prec $ is the lexicographic order on the set  of index sets in \red{$\mathcal{I}$} of cardinality $r+1$. 
For $0\le l\le k-1$, consider the exact sequences of sheaves
\[
0\to F_I-(l+1)E_I\to F_I-lE_I\to {(F_I-lE_I)}_{|E_I}\to0.
\]

By Lemma \ref{lemma-EdotE} and Remark \ref{remark-intersection}, we have
\begin{equation}\label{equation-third term}
 {(F_I-l E_I)}_{|E_I}\cong \OO_{X^r_{(r-1)}\times \PP^{n-r-1}}((l-k) x, l)
\end{equation}
where $x=x_{I}$ is defined by \eqref{def-x}.

Now let us compute the dimension of the cohomology groups of \eqref{equation-third term}.
Clearly we have 
$h^i(\OO_{\PP^{n-r-1}}( l))=0$ for all $i\neq 1$ and
$$h^{0}(\OO_{\PP^{n-r-1}}( l))={{n-r-1+l}\choose l}.$$

In order to compute $h^i(\OO_{X^r_{(r-1)}}((l-k)x))$, notice that, 
by Serre duality and  \eqref{canonical}, we have
$$h^i(\OO_{X^r_{(r-1)}}((l-k)x))=h^{r-i}(\OO_{X^r_{(r-1)}}((k-l-1)x-h )).$$

Set $a:=(k-l-1)$. Notice that $0\le a\le k-1$ and consider
on $X^r_{(0)}$ the following divisor:
$$y=(ar-1) h-\sum_{j=1}^{r+1} a(r-1) e_j.$$
By using Proposition \ref{sharpBLL}
one easily computes the strict transform of $y$ via 
$\pi=\pi_{(r-1)}^r\circ\cdots\circ\pi_{(1)}^r$:
$$\widetilde{y}=(ar-1) h- \sum a(r-1)  e_j- \sum (a(r-2)+1)  e_{ij}-
\ldots-\sum (r-1) e_{J(r-1)}.
$$
Since
$$
ax -h =\widetilde{y}+\left(\sum e_{ij}+\ldots +\sum (r-1) e_{J(r-1)}\right),
$$ by $(B_{r,r+1,r-1})$ we conclude that 
$h^i(ax -h )=h^i(\widetilde{y})$, for $i\ge 0$.
On the other hand, as $(C_{r,r+1,r-1})$ holds for $y$, then 
 $h^i(\widetilde{y})=0$ for all $i\ge 1$.
Furthermore one computes 
\[
h^0(\widetilde{y})=\binom{(a+1)r-1}{r}
+\sum_{\substack{I(\rho),\\ 0\leq  \rho\leq r-1}} 
(-1)^{\rho-1}\binom{(a+1)(r-\rho-1)+\rho}{r}
=\binom{a}{r},
\] 
obtaining the first equality by a combined application of statements
$(A_{r,r+1,\rho})$, $(B_{r,r+1,\rho})$,  and $(C_{r,r+1,r-1})$, that hold true for $y$, 
 and the second equality 
 by induction on $r$.
It follows that $h^i((l-k)x)=0$ for all $i\neq r$, and that 
and $h^r((l-k)x)=0$ for $k-r\le l \le k-1$, while
$h^r((l-k)x)=\binom{k-l-1}{r}$ for $0\le l \le k-r-1$. 

By means of K\"unneth formula 
(see e.g.\ \cite[Section 2.1.7]{shafarevich-II}), one easily computes that the only
non-vanishing cohomology group of the sheaf introduced in \eqref{equation-third term} is
\[
h^r((F_I-l E_I)_{|E_I})=\binom{k-l-1}{r}{{n-r-1+l}\choose l} 
\]
for $0\le l\le k-r-1$. 
From this we derive immediately statement $(B_{n,s,r})$ and the fact that
 $h^i(D_{(r)})=h^i(D_{(r-1)})$, for $i\ne r,r+1$ and for every $1\le r\le \min(n,s)-1$. 
In particular we obtain, for $i\ge r+2$, that
 $h^i(D_{(r)})=h^i(D_{(r-1)})=\cdots h^i(D_{(0)})=0$, the last equality being trivially true.
Finally from 
\[
\sum_{l=0}^{k-r-1}{{k-l-1}\choose r}{{n-r-1+l}\choose l}={{n+k-r-1}\choose n},
\]
 which can be easily proved by induction, we get 
$h^r(D_{(r-1)})=h^r(D_{(r)})-h^{r+1}(D_{(r)}) +
\sum{{n+k_{I(r)}-r-1}\choose{n}}$, where the summation 
ranges over all the multi-indices $I(r)\subseteq\{1,\dots,s\}$.
This completes the proof of $(A_{n,s,r})$.
\end{proof}

\begin{remark}\label{simple lines, double planes} 
The geometric meaning of $(B_{n,s,r})$ is that if a linear $r$-cycle is contained with multiplicity
 at most $r$ in the base locus of a linear system $\ls$, then it does not contribute to the speciality 
of $\ls$ (see Example \ref{sextics in P4}).
\end{remark}

\begin{proposition}\label{normal bundle implies toric and n+2}
Let $n\ge 2$, $s\ge0$ and $1\le \bar{r}\le\min(n-1,s)-1$.
If statements $(A_{n,s,\rho})$ and $(B_{n,s,\rho})$ hold for any $1\le \rho\le \bar{r}$ and moreover 
statements $(C_{n,s,R-1})$ and $(C_{n-1,s-1,R})$ hold for any $1 \le R\le\bar{r}$, then
 statement $(C_{n,s,\bar{r}})$ holds. 
\end{proposition}

\begin{proof}
Recall that  $\LL=\LL_{n,d}(m_1,\dots,m_s)$ satisfies condition
 (\ref{condition-mult}), by  Lemma \ref{effectivity}.

Set $\sigma:=\min(n,s)$. Define the linear system
$\ls':=\LL_{n-1,d}(m_1,\dots,m_\sigma)$  in $\mathbb{P}^{n-1}$ and notice 
that it also satisfies (\ref{condition-mult}). 
Moreover define $\hat{\ls}$ to be the linear system in $\mathbb{P}^n$ given either by 
$\hat{\ls}:=\LL_{n,d-1}(m_1-1,\dots,m_{\sigma}-1)$ if $s\le n$, or by
$\hat{\ls}:=\LL_{n,d-1}(m_1-1,\dots,m_{n}-1,m_{n+1})$ if $s=n+1$, or by
$\hat{\ls}:=\LL_{n,d-1}(m_1-1,\dots,m_{n}-1,m_{n+1},m_{n+2})$ if $s=n+2$. 
It also satisfies (\ref{condition-mult}), for all $s\le n+2$.

Let now
\begin{align*}
D_{(\bar{r})}&=dH-\sum_{i=1}^{s} m_i E_i-\sum_{\substack{I(\rho),\\1\le \rho\le \bar{r}}}k_{I(\rho)}E_{I(\rho)},
\\
H_{(\bar{r})}&=H-\sum_{i=1}^{\sigma} E_{i}-\sum_{\substack{I(\rho),\\I(\rho)\subseteq\{1,\dots,\sigma\}, 
\\1\le \rho\le \bar{r}}} {E}_{I(\rho)}
\end{align*}
be the strict transforms respectively of $\ls$ and of a hyperplane containing
 the points $p_1,\dots,p_{\sigma}$ in $X^n_{(\bar{r})}$ and
consider the following Castelnuovo exact sequence:
\begin{equation}\label{castelnuovo for n+1 n+2}
0\to D_{(\bar{r})}-H_{(\bar{r})}\to D_{(\bar{r})}\to {D_{(\bar{r})}}_{|H_{(\bar{r})}}\to 0.
\end{equation}
We can identify the restricted divisor ${D_{(\bar{r})}}_{|H_{(\bar{r})}}$ 
of the sequence (\ref{castelnuovo for n+1 n+2})
with the strict transform $D'_{(\bar{r})}$ of the general divisor in $\LL'$. 
Moreover if $\hat{D}_{(\bar{r})}$ is the strict transform of a general divisor in $\hat{\ls}$
 in $X^n_{(\bar{r})}$, then if $s\le n+1$ we can identify the divisors 
$D_{(\bar{r})}-H_{(\bar{r})}$ and $\hat{D}_{(\bar{r})}$,
while if $s=n+2$ we have
\[
D_{(\bar{r})}-H_{(\bar{r})}=\hat{D}_{(\bar{r})}+ 
\sum_{\substack{I(\rho)\\\{n+1,n+2\}\subseteq I(\rho),\\ 1\le \rho\le \bar{r}}}E_{I(\rho)}.
\]
From  $(B_{n,s,\rho})$, $1\le \rho\le \bar{r}$, we derive the equalities
$h^i(D_{(\bar{r})}-H_{(\bar{r})})=h^i(\hat{D}_{(\bar{r})})$,  $i\ge 0$.
Notice that the dimension of the linear base locus of  of $\hat{\ls}$ is $\le \bar{r}$.
We iterate the procedure, each time reordering the points with respect to their multiplicity from  the 
highest to the lowest, until   the kernel  corresponds to a divisor whose linear base
 locus has dimension $\le \bar{r}-1$.   
We conclude
that $h^i(D_{(\bar{r})})=0$,  for $i\ge 1$,
by using $(C_{n,s,R-1})$ and $(C_{n-1,s-1,R})$, with $R\le \bar{r}$,

Finally, if $r<\bar{r}$, then statements $(A_{n,s,\rho})$, with $1\le \rho\le \bar{r}$, imply
 $h^i(D_{(r)})=0$, for $i\ge 1$, $i\ne r+1$. 
Indeed if $1\le i\le r$, from $(A_{n,s,\rho})$, $1\le \rho\le r$, 
one deduces the equalities $h^{i}(D_{(r)})=h^{i}(D_{(r+1)})=\cdots=h^i(D_{(\bar{r})})=0$;
while, on the other hand,  if $i\ge r+2$ then statements $(A_{n,s,r})$ implies
 $h^{i}(D_{(r)})=0$.
If $r>\bar{r}$ then $D_{(r)}$ is the total transform in $X_{(r)}$ of $D_{(\bar{r})}$, 
therefore  $h^i(D_{(r)})=h^i(D_{(\bar{r})})=0$, $i\ge 1$.
\end{proof}

\begin{proof}[Proof of Theorem \ref{monster-theorem}]
We will prove  statements $(A_{n,s,r})$, $(B_{n,s,r})$, 
with $n\ge1 , s\ge 0$ and $1\le r\le\min(n,s)-1$,
 by induction on $n$,  and statement $(C_{n,s,R})$, 
with $n\ge1 , s\ge 0$ and $0\le R\le\min(n,s)-1$, 
by induction on $n$ and $R$.

If $n=1$, then statements $(A_{1,s,r})$ and $(B_{1,s,r})$ trivially hold;
furthermore statement $(C_{1,s,R})$ holds as well because
 any linear system in $\PP^1$ is non-special, namely 
$h^i(X^1_{(0)},D_{(0)})=0$,  $i\ge 1$.

In order to prove $(A_{n,s,r})$ for $n\ge 2$ and $1\le r\le \min(n,s)-1$,
 we may assume by induction on $n$ 
that $(A_{r,r+1,\rho})$, $(B_{r,r+1,\rho})$, 
for any $1\le \rho\le r-1$ and $(C_{r,r+1,r-1})$  hold so that
we can apply Proposition \ref{toric implies normal bundle} to obtain $(A_{n,s,r})$ 
and $(B_{n,s,r})$.
This proves part $(i)$.

Now we prove $(C_{n,s,R})$ by induction on $R$.
If $R=0$ then $\LL$ contains only points in its base locus and it is well-known 
to be non-special for $n\ge 1$ and $s\le n+2$, 
hence $(C_{n,s,0})$ holds;
while if $R=n-1$, we reduce  to the case $R\le n-2$ exploiting Remark \ref{hyperplanes in base locus} . 
In order to prove $(C_{n,s,R})$ for the pair $(n,R)$, 
with $n\ge 2$ and $1\le R\le \min(n-1,s)-1$, we 
assume that $(A_{n,s,\rho})$ and $(B_{n,s,\rho})$, for any $\rho$ with
 $1\le \rho\le R-1$, and $(C_{n-1,s-1,R})$, $(C_{n,s,R-1})$ hold.
By applying Proposition \ref{normal bundle implies toric and n+2}
we get $(C_{n,s,R})$ and this complete the proof of part $(ii)$.
\end{proof}

\subsection{Examples}

\begin{example}
Consider the case $\LL=\LL_{n,3}(3^3)$, with $n\ge 3$. Notice that the dimension of
the linear base locus of $\LL$ is $\bar{r}=2$. Indeed,
by Proposition \ref{sharpBLL},
all the lines and the plane spanned by the three base points are triply contained 
in the base locus of $\LL$, i.e. $k_{12}=k_{13}=k_{23}=k_{123}=3$.
Let us compute the cohomologies of the strict transforms $D_{(r)}$, $r\ge 0$, by
means of Theorem \ref{monster-theorem}: 
\begin{itemize}
\item $h^i(X^n_{(0)},D_{(0)})=0$, for $i\ge 2$ and $h^1(X^n_{(0)},D_{(0)})
=-h^2(X^n_{(1)},D_{(1)})+3(n+1)$;
\item $h^i(X^n_{(1)},D_{(1)})=0$, for $i\ge 1, i\ne 2$ and $h^2(X^n_{(1)},D_{(1)})=1$;
\item $h^i(X^n_{(r)},D_{(r)})=0$, for $i\ge 1, r\ge 2$. 
\end{itemize}
Therefore  $h^1(\PP^n,\LL)=h^1(X^n_{(0)}, D_{(0)})=3n+2$ and 
one can hence easily compute
$h^0(\PP^n,\LL)=h^0(X^n_{(r)},D_{(r)})={n\choose 3}$, as 
it was already obtained in Proposition \ref{cones-theorem}.
\end{example}

\begin{example}\label{sextics in P4}
Consider the linear system $\LL_{4,6}(5^3,4,3,2)$. It has virtual dimension
$\vdim(\LL)=-56$ and  linear virtual (and expected) dimension  $\ldim(\LL)=6$. 
The linear
base locus of $\LL$ is formed by multiple points, lines and planes and a
 $3$-dimensional linear cycle, 
by Proposition \ref{sharpBLL},
 namely  $\bar{r}=3$.

Computing the values of the integers $k_{I(r)}$, for $1\le r\le \bar{r}$, 
we see that the contributions of the multiple lines 
to the speciality of $\LL$ is $63$ and that only the plane through the three points of
multiplicity $5$, which is triply contained in the base locus, gives a correction equal to $-1$.
Moreover the $3$-dimensional cycle is simply contained in the base locus so, as we 
noticed in Remark \ref{simple lines, double planes}, it does not create speciality.
Exploiting Theorem \ref{monster-theorem}, we get
\begin{itemize}
\item  $h^1(X^n_{(0)},D_{(0)})
=-h^2(X^n_{(1)},D_{(1)})+63$;
\item $h^2(X^n_{(1)},D_{(1)})=1$;
\end{itemize}
and all the other cohomologies vanish. Therefore $\LL$ is linearly non-special 
and $h^1(\PP^4,\LL)=62$ and $h^0(\PP^4,\LL)=6$.
\end{example}

\section{Linear systems with more than $n+2$ points}
\label{n+3-section}

In this section we will obtain a sufficient condition to be linearly non-special for a linear system 
$\ls_{n,d}(m_1,\dots,m_s)$ 
when $s\ge n+3$.

\begin{lemma}\label{colorful} 
Given integers n,d, $m_i\le d$, consider the two linear systems 
$\ls=\ls_{n,d}(d,m_1,\dots,m_s)$ and
$\ls'=\ls_{n-1,d}(m_1,\dots,m_s)$. 
If $\LL$ and $\LL'$ are non-empty, then we have $\dim(\ls)=\dim(\ls')$, 
and $\ldim(\ls)=\ldim(\ls')$.
\end{lemma}
\begin{proof}
The first equality is obvious since any divisor in $\ls$ is a cone with vertex in the point of multiplicity $d$.
The second equality is easily proved. 
\end{proof}

\begin{remark} \label{additivity-odim} 
Given $\ls=\ls_{n,d}(m_1,\dots,m_s)$ and $1< t< s$, 
assume that the set $I=\{(i,j):t+1\le i<j\le s, m_i+m_j-d>0\}$ contains at most two pairs 
$(i_1,j_1),(i_2,j_2)$.
Consider
$\ls'=\ls_{n-1,d}(m_1,\ldots,m_t,m_{i_1}+m_{j_1}-d, m_{i_2}+m_{j_2}-d)$ and
$\ls''=\ls_{n,d-1}(m_1-1,\ldots,m_{t}-1,m_{t+1},\ldots,m_s)$.
If $\LL,\LL',\LL''$ are non-empty, then it is easy to prove that
$$\ldim(\ls)=\ldim(\ls')+\ldim(\ls'')+1.$$
\end{remark}

The following theorem is the main result of this section:
\begin{theorem}\label{theorem-n+3}
Given the integers $n\geq1$, $d\geq 2$, $s\geq n+3$, $d\ge m_1\ge \ldots\ge m_s\ge 1$, consider  the linear system $\ls=\ls_{n,d}(m_1,\dots,m_s)$.
Let $s(d)\geq0$ be the number of points of multiplicity $d$, that is $m_{s(d)}= d$ and $m_{s(d)+1}\le d-1$.
Define $$b(\ls):=\min \{n-s(d),s-n-2\}.$$ 
Then we have  $\dim(\ls)=\ldim(\ls)$, 
if the following condition is satisfied:
\begin{equation}
\label{EffectivityCondition}
\sum_{i=1}^s m_i\leq nd+ b(\ls)
\end{equation}
\end{theorem}

We split the proof of the theorem in the two Lemmas \ref{lessorequal} and \ref{greaterorequal}.

\begin{lemma}\label{lessorequal}
Let $\ls=\ls_{n,d}(m_1,\dots,m_s)$ such that \eqref{EffectivityCondition} is satisfied. Then
$$\dim(\ls)\le \ldim(\ls).$$
\end{lemma}

\begin{proof}
Notice that it is enough to prove the statement when equality holds in 
\eqref{EffectivityCondition} for a collection of $s$ general points. 
Indeed any collection $Z$ of $s$ fat points
 which strictly satisfies inequality
 \eqref{EffectivityCondition} 
can be thought as  subscheme of a collection $Z'$ of fat points given by $Z$ and 
a suitable number $s'$ of simple (general) points,
which satisfies equality in \eqref{EffectivityCondition}; moreover if $\ls'\subset \ls$ 
is the linear system formed by the divisors in $\ls$ that pass through these extra simple 
points, then $\dim(\ls')=\dim(\ls)-s'$ and $\ldim(\ls')=\ldim(\ls)-s'$.

We prove the statement by induction on $d$ and $n$
 based on the cases 
$n=1$ and $d=2$, which are easily checked.

If $s(d)\geq1$, then applying Lemma \ref{colorful}  $s(d)$ times we obtain 
$\dim(\ls)=\dim(\ls')$ and $\ldim(\ls)=\ldim(\ls')$, where
$\ls':=\ls_{n-s(d),d}(m_{s(d)+1},\dots,m_s)$.
Notice moroever that $b(\ls')=b(\ls)$, and that from \eqref{EffectivityCondition} 
it immediately follows that 
$$\sum_{i=s(d)+1}^sm_i\le (n-s(d))d+b(\ls'),$$
hence by induction we have $\ldim(\ls')=\dim(\ls')$, so
$\ldim(\ls)=\dim(\ls)$.

Assume  $s(d)=0$. 
For any pair $(n,d)$, we will assume  the statement true for $(n-1,d)$ as well as for $(n,d-1)$ and consider suitable specializations of the points on a hyperplane $H\subset\PP^n$ in such a way that, in the so obtained Castelnuovo exact
sequence
 $$0\to \hat{\ls}\to \ls\to \ls_H\to 0,$$ 
both the restricted system $\ls_H$ and the kernel system $\hat{\ls}$ satisfy the hypotheses of the theorem. 
We split the proof in the following subcases:
\begin{itemize}
\item[(i)] $m_1+m_2\leq d$;
\item[(ii)] $m_1+m_2\ge d+ 1$ and $b(\ls)= s-n-2<n$; 
\item[(iii)] $m_1+m_2\geq d+ 1$, $m_1+m_s\ge d+1$ and $b(\ls)=n$;
\item[(iv)] $m_1+m_2\geq d+1$, $m_1+m_s\le d$  and $b(\ls)=n$. 
\end{itemize} 

\smallskip

Case (i).
Consider the hyperplane $H$ through the first $n$ points 
and notice that the trace $\ls_H$ satisfies
$\dim(\ls_H)\le \dim( \ls_{n-1,d}(m_1,\dots, m_n))$.
From the assumption it follows that
$$\sum_{i<j, i,j=1}^n (m_i+m_j-d)=(n-1)\sum_{i=1}^n m_i-{n\choose2}d\leq0.$$
Therefore, since $n\geq 2$, we have 
$\sum_{i=1}^n m_i\leq \frac{n}{2}d\leq (n-1)d$, so the linear system is not empty by 
Lemma \ref{effectivity}.
Hence by Corollary \ref{corollary-monster} 
we have $\dim(\ls_{n-1,d}(m_1,\dots, m_n))=\ldim(\ls_{n-1,d}(m_1,\dots, m_n))$. 
The residual is $\hat{\ls}=\ls_{n,d-1}(m_1-1,\cdots,m_n-1,m_{n+1}\dots,m_s)$. 
Since $d\ge 3$ and $m_1+m_2-d\leq0$ we have that there are no points of multiplicity $d-1$ in $\hat{\ls}$. 
If $m_n\geq 2$, then it is based on $s$ points, and we have $b(\hat{\ls})= b(\ls)$. 
Otherwise, let
 $l\in\{1,\dots,n\}$ be the integer such that $m_{n-l+1}=\cdots=m_s=1$.  
As $b(\ls)=\sum_{i=1}^sm_i-nd\le (n-l)(d-1)+(s-n+l)-nd=s-2n-(d-2)l
\le s-n-2-l$, then we have $b(\hat{\ls})=b(\ls)=n$.
In both cases from \eqref{EffectivityCondition} it follows that
$$\sum_{i=1}^n (m_i-1)+\sum_{i=n+1}^{s}{m_i}= n(d-1)+b(\hat\ls),$$ 
hence by induction we obtain $\dim(\hat{\ls})=\ldim(\hat{\ls})$.
We conclude $\dim(\ls)\leq \dim(\ls_H)+\dim(\hat\ls)+1 \leq 
\ldim(\ls_{d}^{n-1}(m_1,\dots, m_n))+\ldim(\hat\ls)+1=\ldim(\ls)$
where the last equality follows from Remark \ref{additivity-odim}.

\smallskip 

Case (ii). 
In this case we specialize the last $s-2$ points on a general hyperplane $H$.
The trace is $\ls_H=\ls_{n-1,d}(m_1+m_2-d,m_3,\dots,m_s)$.
From the assumption it follows
$b(\ls_H)=\min(n-1,(s-1)-(n-1)-2)=b(\ls)$
and so we have
$$(m_1+m_2-d)+\sum_{i=3}^s m_i \le (n-1)d+b(\ls_H),$$
and we get $\dim(\ls_H)=\ldim(\ls_H)$ by induction.  
The residual is $\hat{\ls}=\ls_{n,d-1}(m_1,m_2,m_3-1,\dots,m_s-1)$.
From \eqref{EffectivityCondition} we have
$$m_1+m_2+\sum_{i=3}^s (m_i-1)= nd+b(\ls)-(s-2)=n(d-1).$$
Let $\hat s$ the number of points where $\hat\ls$ is supported.
We conclude that $\dim(\hat{\ls})=\ldim(\hat{\ls})$, using 
Corollary \ref{corollary-monster}
if $\hat s\le n+2$, and using induction if $\hat s\ge n+3$.
As in the previous case, by Remark \ref{additivity-odim}, we conclude 
$\dim(\ls)\leq \dim(\ls_H)+\dim(\hat\ls)+1 =
\ldim(\ls_H)+\ldim(\hat\ls)+1=\ldim(\ls).$

\smallskip

Case (iii).
Notice that $m_{s-1}+m_s\leq d$. 
Indeed if $m_i+m_j\geq d+1$ for all $1\le i,j\le s$, then 
$(s-1)\sum_{i=1}^s m_i\geq {s\choose 2}(d+1)$, so that 
$\sum_{i=1}^s m_i\ge \frac{s(d+1)}2 \ge (n+1)(d+1)> nd+n$,
 and this leads to a contradiction with (\ref{EffectivityCondition}).

Since $m_1+m_s-d\ge 1$, 
then there exists an integer  $\iota\in\{1,\dots,s-2\}$ such that
$m_\iota+m_s\geq d+1$, and $ m_{\iota+1}+m_s\leq d.$
Let us specialize all points but $m_{\iota},m_{\iota+1},m_{s}$ on a hyperplane $H$.
Since $m_{\iota}+m_{\iota+1}\ge m_{\iota}+m_{s}\ge d+1$, it follows that
the trace is
$\ls_H=\ls_{d}^{n-1}(m_1,\dots,m_{\iota-1},m_{\iota+2}\dots,m_{s-1},m_{\iota}+m_{\iota+1}-d, m_{\iota}+m_{s}-d)$
and $b(\ls_H)=\min (n-1, (s-1)-(n-1)-2)=n-1$.
Since $m_\iota<d$ and $b(\ls_H)=n-1$
then
$$\sum_{i=1}^s m_i+m_\iota-2d\le (n-1)d+n+(m_\iota-d)\le (n-1)d+b(\ls_H),$$
and thus we have $\dim(\ls_H)=\ldim(\ls_H)$ by induction.  
The residual linear system is 
$\hat{\ls}=\ls_{n,d-1}(m_1-1,\ldots,m_{\iota-1}-1,m_\iota,m_{\iota+1},m_{\iota+2}-1,\dots,m_{s-1}-1,m_s).$
Note that from the assumptions $m_\iota+m_s\ge d+1$ and $s(d)=0$ 
we have $m_s\ge2$, hence $\hat\ls$ is supported at $s$ points and $b(\hat\ls)\ge1$. 
It easily follows, being $s\ge 2n+2$ by assumption
that
$$\sum_{i=1}^{s} m_i-(s-3)= nd+n-s+3\le n(d-1)+1\le n(d-1)+b(\hat\ls),$$
 and we conclude that $\dim(\hat{\ls})=\ldim(\hat{\ls})$  by induction. 
We conclude, like in the previous cases, that $\dim(\ls)\le\ldim(\ls)$.

\smallskip

Case (iv). 
We specialize all points but $m_1,m_2,m_{s}$ on a hyperplane $H$.
The trace is
$\ls_H=\ls_{d}^{n-1}(m_1+m_2-d, m_3,\dots,m_{s-1})$
and $b(\ls_H)=\min (n-1, (s-2)-(n-1)-2)=b(\ls)-1=n-1$.
Since
$$m_1+m_2-d+\sum_{i=3}^{s-1} m_i\le nd+b(\ls)-m_s-d \le (n-1)d+b(\ls_H),$$
we have $\dim(\ls_H)=\ldim(\ls_H)$ by induction.  
The residual linear system is 
$\hat{\ls}=\ls_{n,d-1}(m_1,m_2,m_3-1,\ldots,m_{s-1}-1,m_s)$, 
and, as in the previous case
$\sum_{i=1}^{s} m_i-(s-3)\le  n(d-1)+1.$
Let $\hat s$ the number of points where $\hat\ls$ is supported. 
If $\hat s\ge n+3$, then clearly $b(\hat \ls)\ge 1$ and we conclude by induction.
If $\hat{s}\le n+2$ then it has to be $m_{s-1}=1$. Therefore the sum 
of the multiplicities of the points in $\hat{\ls}$ is
$\sum_{i=1}^{s-2}m_i-(s-2)+m_s\le \sum_{i=1}^s m_i-(s-2)\le n(d-1)$,
 being $s\ge 2n+2$.
Hence it is non-empty by Lemma \ref{effectivity} and we conclude by Corollary \ref{corollary-monster}.
Finally we conclude, as in the previous cases, that $\dim(\ls)\le\ldim(\ls)$.
\end{proof}

\begin{lemma}\cite[Lemma 6.3]{Chandler}\label{chandler}
Let $k\ge 1$, $d-1\ge m_1\ge \ldots\ge m_s\ge 1$ and consider the linear systems
$\ls=\ls_{n,d}(k,m_1,\dots,m_s)$ and
$\ls'=\ls_{n,d}(k-1,m_1,\dots,m_s)$.
Denote $c_i=\max \{k+m_i-d-1,0\}$ and $\overline{s}=\max \{1\le i\le s: c_i>0\}$
 and
$\ws=\ls_{n-1,k-1}(c_1,\dots,c_{\overline{s}})$. 
Then we have
\[
\dim(\ls)\ge \dim(\ls')-\dim(\ws)-1
\]
\end{lemma}

\begin{remark}\label{remark-chandler}
With the notation of the previous lemma it is easy to check that if $\LL,\LL', \ws$ are non-empty, then
$$\ldim(\ls)=\ldim(\ls')-\ldim(\ws)-1.$$
\end{remark}

\begin{lemma}\label{greaterorequal} 
Let $\ls=\ls_{n,d}(m_1,\dots,m_s)$ such that \eqref{EffectivityCondition} is satisfied. Then
$$\dim(\ls)\ge \ldim(\ls).$$
\end{lemma}
\begin{proof}
We may assume $s(d)=0$, thanks to Lemma \ref{colorful}.

The proof is by induction on $\sum_{i=1}^s m_i$, based on the case
 $\sum_{i=1}^s m_i=s$ for which the statement trivially holds being $\ldim(\ls)=\vdim(\ls)$.

If $\sum_{i=1}^s m_i>s$, by induction we assume
that the statement holds for any subscheme 
strictly contained in $\ls$.
Consider the linear systems  
\[\ls':=\ls_{n,d}(m_1-1,m_2,\dots,m_s) \ \textrm{ and }  \ \ws:=\ls_{n-1,m_1-1}(c_1,\dots,c_{\overline{s}}),\]
where $c_{i}:=\max\{k_{1,i+1}-1,0\}$, for $2\le i \le s$  
and $\overline{s}:=\max \{2\le i\le s: c_i>0\}-1$. 
Clearly $\ls'$ satisfies condition \eqref{EffectivityCondition} and by induction we have $\dim(\ls')\ge\ldim(\ls')$.
We claim that 
\[
\sum_{i=1}^{\bar{s}}c_i\le (n-1)(m_1-1).
\]
It is easily verified when $\bar{s}\le n-1$, while
if $\bar{s}\ge n$, since
\begin{align*}
\sum_{i=1}^{\bar{s}}c_i&=\sum_{i=2}^{\bar{s}+1}(m_1+m_i-d-1)=\sum_{i=2}^{\bar{s}+1} m_i+\bar{s}(m_1-d-1)\\
&\le\sum_{i=1}^sm_i-m_1-(s-\bar{s}-1)+\bar{s}(m_1-d-1),
\end{align*}
then from condition \eqref{EffectivityCondition} and the fact that $b(\ls)\le s-n$ we get
\begin{align*}
\sum_{i=1}^{\bar{s}}c_i&\le nd+b(\ls)-m_1-(s-1)+\bar{s}(m_1-d)\\
&\le(n-1)(m_1-1).
\end{align*}
Now, if $\bar{s}\le n+2$ we already proved in Corollary \ref{corollary-monster}
that $\dim(\ws)=\ldim(\ws)$; if $\bar{s}\ge n+3$ 
then by Lemma \ref{lessorequal} we have that $\dim(\ws)\le \ldim(\ws)$.
Finally, by Lemma \ref{chandler} and  Remark \ref{remark-chandler} we have
$$\dim(\ls)\ge \dim(\ls')-\dim(\ws)-1\ge \ldim(\ls')-\ldim(\ws)-1= \ldim(\ls).$$
\end{proof}

\section{Concluding remarks and future directions}
\label{final-section}

\subsection{The Fr\"oberg-Iarrobino conjecture and Chandler results}
\label{section6.1}
The results of this paper are connected with the Fr\"oberg-Iarrobino conjecture, which
is the geometrical version of an important conjecture formulated by Fr\"oberg
in the commutative algebra setting. 

More precisely, let $R=\CC[x_0,\ldots,x_n]$ be a polynomial ring in $n+1$ variables 
over $\CC$ and $I=(f_1,\ldots,f_s)$ be an ideal generated by $s$ general forms of degrees $m_1,\ldots,m_s$. In \cite{Froberg}, 
Fr\"oberg conjectures that the Hilbert series of the quotient ring $R/I$ is
\begin{equation}\label{froberg-formula}
\HHilb_{R/I}(t)
=\left[\frac{\Pi_{i=1}^{s}(1-t^{m_i})}{(1-t)^{n+1}}\right]
\end{equation}
where we use the notation: 
$[\sum a_it^i]=\sum b_it^i$ with $b_i=a_i$ if $a_j>0$ for all $j\le i$, and $b_i=0$ otherwise.

The Fr\"oberg conjecture has been proved to be true for $n=1$ \cite{Froberg} and $n=2$ \cite{Anick}.  
Moreover, it is easily verified 
when the number of generator is $s\le n+1$, and in
the case $s=n+2$ has been proved by Stanley (see \cite[Example 2]{Froberg}). 
The conjecture is known to be true in some other special cases but it is still open in general.

The Fr\"oberg-Iarrobino conjecture concerns the case when every form $f_i=(l_i)^{m_i}$
is a power of a general linear form $l_i$ (see \cite{Chandler, Iarrobino} for more details).
The Fr\"oberg-Iarrobino conjecture implies the Fr\"oberg conjecture,
but they are not equivalent, because a 
power of a general linear form is not a general form.

More precisely, the Fr\"oberg-Iarrobino conjecture deals with the homogeneous case, i.e. $m_i=m$, $i=1,\dots,s$,
and states that formula \eqref{froberg-formula} 
holds except for a given list of exceptions (see \cite[Conjecture 4.8]{Chandler}).
It is natural to generalize such a conjecture to the non-homogeneous case, namely when the $m_i$'s 
are different, as Chandler pointed out. 
Since in the case of powers of linear forms, the ideal $I$ can be seen as the ideal of a collection of fat points,
it is possible to give a geometric interpretation of such a conjecture in terms of our Definition \ref{new-definition}, 
namely a linear system is always linearly 
non-special but in a list of exceptions.

In \cite[Proposition 9.1]{Chandler} it is proved that the the generalized Fr\"oberg-Iarrobino conjecture 
is true if either $s\le n+1$ or $\sum_{i=1}^sm_i\le dn+1$.
Our Corollary \ref{corollary-monster}
and Theorem \ref{theorem-n+3} improve Chandler's result and show that the generalized 
 Fr\"oberg-Iarrobino conjecture holds true if
either $s\le n+2$ or condition \eqref{EffectivityCondition} is satisfied.

As already mentioned in Remark \ref{old5.1} there exists a weak version of the 
Fr\"oberg-Iarrobino conjecture, see \cite[Conjecture 4.5]{Chandler}, which states 
that for any linear system $\LL$ the inequality $\ldim(\LL)\le \dim(\LL)$ is verified.
The weak Fr\"oberg-Iarrobino conjecture for the case of $\PP^n$,  $n\le 3$, is
established in \cite[Theorem 1.2]{Chandler}.
Moreover in Lemma \ref{greaterorequal} we prove that such a conjecture holds true for any 
$n$ and arbitrary number of points if condition \eqref{EffectivityCondition} is satisfied.

\subsection{The Laface-Ugaglia conjecture and future directions}\label{LU conjecture}
\label{section6.2}

In view of extending the
well-known Segre-Harbourne-Gimigliano-Hirschowitz conjecture to $\PP^3$,
Laface and Ugaglia,  in 
\cite[Conjecture 4.1]{laface-ugaglia-TAMS} and
\cite[Conjecture 6.3]{laface-ugaglia-standard},
formulated the following conjecture.

\begin{conjecture}[Laface-Ugaglia]\label{LU}
If $\ls=\ls_{3,d}(m_1,\dots,m_s)$ is Cremona reduced, i.e. $2d\geq m_{i_1}+m_{i_2}+m_{i_3}+m_{i_4}$, for any $\{i_1,i_2,i_3,i_4\}\subseteq\{1,\dots, s\}$, 
then $\ls$ is special if and only if one of the following holds:
\begin{enumerate}
\item there exists a line $L=\langle p_i, p_j\rangle$, for some $i,j\in\{1,\dots,s\}$ such that $\LL \cdot L\leq -2$;
\item there exists a quadric $Q = \LL_{3,2}(1^9)$ such that $Q\cdot (\LL-Q)\cdot(\LL-K_{\PP^3}) < 0$.
\end{enumerate}
\end{conjecture}

The Laface-Ugaglia conjecture is known to be true 
when the number of points is at most $8$ \cite{volder-laface}, and
when the points are at most quartuple \cite{ballico-brambilla,dumnicki-Iagel}
or quintuple \cite{ballico-brambilla-caruso-sala}.

We remark that this conjecture can be reformulated, according to our definition, saying 
that a Cremona reduced linear system  in $\PP^3$ either is 
linearly non-special, 
or contains in its base locus a quadric surface which gives speciality.

Our Corollary \ref{corollary-monster} proves that if the points are $s\le n+2$, then such a conjecture
holds true also in $\PP^n$ for any $n$,
i.e. a non-empty linear system (not necessarily Cremona reduced) with at most $n+2$ points
is special if and only if $\ls \cdot L\le -2$ for some line $L=\langle p_i, p_j\rangle$.
Moreover we prove that, in this case, $\dim(\ls)=\ldim(\ls)>\edim(\ls)$. 

When the points are $s\ge n+3$, from Theorem \ref{theorem-n+3} it follows
that the same is true under the assumption \eqref{EffectivityCondition}.
Such an assumption is, in particular, a sufficient condition 
for the base locus to contain no multiple rational normal curves. 
In fact, we expect that when multiple rational normal curves appear in the base locus, they 
give a contribution to the speciality of the system.

We point out that a Cremona reduced linear system  in $\PP^3$ does not contain
rational normal curves
in its base locus, but this fact is no longer true in $\PP^4$.
Indeed consider the following example:
\begin{example}\label{example-easy}
Set $\LL=\LL_{4,10}(6^7)$.
Then $\LL$ is Cremona reduced, but its
base locus contains the double rational curve through the seven base points.
On can easily see that $\LL$ is linearly special, since
$\dim(\LL)=140$,
while
$$\ldim(\LL)=\binom{14}4-7\binom94+21-1=139.$$
It seems natural then to think that the double rational normal curve gives exactly a contribution of $1$, 
similarly to a double line.
\end{example}

The following example 
suggests that only the contribution of rational normal curves 
is important, even when there are other multiple curves in the base locus.
\begin{example}\label{example-giorgio} 
Consider the linear system $\LL=\LL_{5,6}(4^9)$.
It has dimension $\dim(\LL)=2$. Indeed there is a normal elliptic curve $C$ 
of degree $6$ through the $9$ points. The secant variety $\sigma_2(C)$ is a threefold 
cut out by two cubics hypersurfaces $\Sigma_1$ and $\Sigma_2$, defined by the equations $F_1=0$ and $F_2=0$. 
Then the equations $F_1^2=0$, $F_1F_2=0$, $F_2^2=0$ define 
three independent hypersurfaces of degree $6$ which have multiplicity $4$ along $C$ and which generate $\HH^0(\LL)$.

On the other hand, taking into account the contributions of the $36$ double lines in the base locus, one computes
$$\ldim(\LL)=\binom{11}5-9\binom{8}{5}+36-1=-7,$$
hence $\LL$ is linearly special.
Furthermore, it is possible to prove that 
the base locus of $\LL$ contains the $9$ double rational normal curves through each set of $8$  points and, 
assuming that each of their contribution is the same as the contribution of a double line (that is $1$), one gets exactly:
$\ldim(\LL)+9=2=\dim(\LL).$
\end{example}

The above examples, besides their intrinsic interest, show that the existence in the base locus of multiple rational normal curves, 
that are not removable by Cremona transformations,  plays an important role.
Hence it seems natural to extend
our definition of dimension of a linear system, 
taking into account also the contribution of such curves.
We plan to develop furtherly these ideas.

On the other hand, as already noticed by Laface and Ugaglia in $\PP^3$,
also the existence of quadric hypersurfaces passing through $9$ general points in the base locus 
can give contribution to the speciality of a linear system.
In this case it does not seem very clear 
how to quantify such contribution.
Consider for example the following list, which contains all the special 
linear systems in $\PP^3$ of degree $10$ and with at most quintuple
points (see \cite[Table 4]{ballico-brambilla-caruso-sala}).
In the columns we write, respectively, the expected dimension, the dimension, the expected base locus,
the residual  system (which is non-special), and the difference between the dimension and the virtual dimension.

\medskip

\begin{tabular}{|c|c|c|c|c|c|c|}
\hline    
 & $\mathrm{edim}$ & $\dim$ & base locus& residual& $\dim-\vdim$\\
\hline
$\LL_{3,10}(5^9)$& $-1$ & 0 & $5Q$& $\emptyset$&30\\
$\LL_{3,10}(5^8,4)$&  $-1$& 1&$4Q$&$\LL_{3,2}(1^8)$&16\\
$\LL_{3,10}(5^8,3,2)$&  $-1$ & 0&$3Q_1\cup 2Q_2$&$\emptyset$&9\\
$\LL_{3,10}(5^8,3)$&  $-1$& 2&$3Q$&$\LL_{3,4}(2^8)$&7\\
$\LL_{3,10}(5^8,2^2)$&  $-1$& 1& $2Q_1\cup 2Q_2$ &$\LL_{3,2}(1^8)$&4\\
$\LL_{3,10}(5^8,2)$& 1 & 3&$2Q$&$\LL_{3,6}(3^8)$&2\\
$\LL_{3,10}(5^7,4^2,2)$&  $-1$& 1&$2Q$&$\LL_{3,6}(3^7,2^3)$&5\\
$\LL_{3,10}(5^7,4,3^2)$& 0  & 1&$Q_1\cup Q_2$&$\LL_{3,6}(3^7,2^3)$&1\\
$\LL_{3,10}(5^7,4^2)$&0  & 5&$2Q$& $\LL_{3,6}(3^7,2^2)$&5\\
\hline
\end{tabular}

\medskip

Special linear systems which contain fixed quadrics 
through simple points appear also in $\PP^4$
(for example $\LL_{4,4}(2^{14})$ and $\LL_{4,6}(3^{14})$),
but quite surprisingly
we are not able to find examples in $\PP^n$ for $n\ge5$.
Understanding better this phenomenon is another goal 
of our future work.

\subsection{Divisors on $\overline{M_{0,n}}$}
\label{section6.3}

Let $\overline{M_{0,n}}$ be the moduli space of stable rational curves with $n$ marked points. 
Kapranov's construction identifies it with a projective variety isomorphic to the projective space $\PP^{n-3}$ 
successively blown up along $L_{I(r)}$, 
$r$-dimensional cycles spanned by $(r+1)$-subsets of a set with $n-1$ general points, 
for $r$ increasing from $0$ to $n-4$. 
In our paper, see Section \ref{up to $n+2$}, $\overline{M_{0,n}}$ is denoted  by $X^{n-3}_{(n-4)}$ for $s=n-1$. 
This space has been well studied, however basic questions are still open.

 We recall here Fulton's conjectures, concerning the description of the Nef cone and the 
Effective cone of $\overline{M_{0,n}}$. Fulton's weak conjecture states that $1$-dimensional boundary strata, 
whose components are called  $F$-curves, generate the Mori cone of curves, 
$\overline{NE}_{1}(\overline{M_{0,n}})$. This is proven for $n\leq 7$ in \cite{KeMc}.
Fulton's strong conjecture, saying that the boundary divisors generate the effective 
cone of $\overline{M_{0,n}}$, is known to be false. For $n=6$ the effective cone is described in
 \cite{HaTs, Ca} as being spanned by boundary divisors and by the
 Keel-Vermeire divisor. Even though pull-backs of the Keel-Vermeire divisor under the
 forgetful morphism are extremal rays that are not boundary divisors, few things are known
 regarding the effective cone of $\overline{M_{0,n}}$ for $n\geq 7$. We should also mention
 the conjecture of Castravet and Tevelev regarding the effective cone of $\overline{M_{0,n}}$
ù for $n\geq 7$, see \cite{CaTe}.

In our paper, Theorem \ref{monster-theorem} for $s=n-1$ points, computes the dimension of all cohomology 
groups for special type of divisors in $\overline {M_{0,n}}$. 
We prove that it depends exclusively on the dimension of the linear base locus, $L_{I(r)}$, the multiplicity of 
containment, $k_{I(r)}$, and $n$, these three information being elegantly encoded in a binomial formula. 
As showed in the examples of Section \ref{LU conjecture}, we expect a similar result to hold 
for divisors interpolating a higher number of points and having non-linear base locus. 
We point out that the divisors we consider in Section \ref{up to $n+2$} live on $\overline{M_{0,n}}$ 
and not on the blow-up of the projective space in points, suggesting that $\overline{M_{0,n}}$ is the 
natural space where interpolation problems on linear cycles, $L_{I(r)}$, 
based on the set with $n-1$ fixed points should be formulated.
For this we believe that our work extends and connects the algebraic approach of Fr\"oberg and 
the geometric perspective of Chandler and Laface-Ugaglia to the geometry of $\overline{M_{0,n}}$.

We believe that interpolating higher dimensional linear cycles, $L_{I(r)}$ for positive $r$, 
is a possible direction for describing the effective cone of $\overline{M_{0,n}}$. 
On the other hand, showing that the $F$-divisors are globally generated by the techniques that we 
developed here, is a possible approach to the F-Nef conjecture.


\begin{thebibliography}{99}

\bibitem{Anick} D.~Anick, {\it Thin algebras of embedding dimension three}, J. Algebra 100 (1986), no. 1, 235--259.

\bibitem{ballico-brambilla} E.~Ballico and M.C.~Brambilla, {\it Postulation of general quartuple fat point schemes in $\PP^3$}, 
J. Pure Appl. Algebra 213 (2009), no. 6, 1002--1012.

\bibitem{ballico-brambilla-caruso-sala} E.~Ballico, M.C.~Brambilla, F.~Caruso, M.~Sala, {\it 
Postulation of general quintuple fat point schemes in $\PP^3$}, J. Algebra 363 (2012), 113--139. 

\bibitem{ale-hirsch} 
 M.C.~Brambilla and G.~Ottaviani, {\it On the Alexander-Hirschowitz theorem},
J. Pure Appl. Algebra 212 (2008), no. 5, 1229--1251. 

\bibitem{CDD} S.~Cacciola, M.~Donten-Bury, O.~Dumitrescu, A.~Lo Giudice, J.~Park, {\it Cones of divisors of blow-ups of projective spaces}, Matematiche (Catania) 66 (2011), no. 2, 153--187.

\bibitem{Ca} A.M.~Castravet, {\it The Cox ring of $\overline{M_{0,6}}$},
Trans. Amer. Math. Soc. 361 (2009), no. 7, 3851--3878.

\bibitem{castravet-tevelev} A.M.~Castravet and J.~Tevelev, {\it Hilbert's 14th problem and Cox rings}, Compos. Math. 142 (2006), no. 6, 1479--1498.


\bibitem{CaTe} A.M.~Castravet and J.~Tevelev, {\it Exceptional loci on $\overline{M_{0,n}}$ and hypergraph curves},  
arXiv:0809.1699 (2008)


\bibitem{Chandler} K.~Chandler, {\it The geometric interpretation of Fr\"oberg-Iarrobino conjectures on infinitesimal neighbourhoods of points in projective space},  J. Algebra 286 (2005), no. 2, 421--455. 

\bibitem{Ciliberto} C.~Ciliberto, {\it Geometrical aspects of polynomial
  interpolation in more variables and
of Waring's problem}, European Congress of Mathematics, Vol. I 
(Barcelona, 2000), 289--316, Progr. Math., 201, Birkh\"auser, Basel, 2001. 


\bibitem{CHMR} C.~Ciliberto, B.~Harbourne, R.~Miranda, J.~Ro\'e,
{\it Variations on Nagata's conjecture}, arXiv:1202.0475 (2012).

\bibitem{volder-laface}
C.~De Volder and A.~Laface, 
{\it On linear systems of $\PP^3$ through multiple points}, J. Algebra 310 (2007), no. 1, 207--217. 

\bibitem{Dolgachev} I.~Dolgachev, {\it Weyl groups and Cremona transformations}, Singularities, 
Part 1 (Arcata, Calif., 1981), 283--294, Proc. Sympos. Pure Math., 40, Amer. Math. Soc., Providence, RI, 1983.


 
\bibitem{Dumnicki} M.~Dumnicki, {\it Regularity and non-emptiness of linear systems in $\PP^{n}$},  arXiv:0802.0925 (2008).

\bibitem{dumnicki-Iagel} M.~Dumnicki, {\it On hypersurfaces in $\PP^3$ with fat points in general position}, Univ. Iagel. Acta Math. No. 46 (2008), 15--19.

\bibitem{Froberg} R.~Fr\"oberg,
{\it An inequality for Hilbert series of graded algebras},
Math. Scand. 56 (1985), no. 2, 117--144.

\bibitem{fulton} W.~Fulton, {\it Intersection theory}, Springer, Berlin 1998.


\bibitem{macaulay}
D.~Grayson, M.~Stillman, {\it Macaulay 2, a software system
for research in algebraic geometry}, available at
http://www.math.uiuc.edu/Macaulay2/.


\bibitem{hartshorne}
R.~Hartshorne, {\it Algebraic geometry}, Springer, New York 1977.

\bibitem{HaTs}B.~Hassett, Y.~Tschinkel,  {\it On the effective cone of the moduli space of pointed rational
curves}, Topology and geometry: commemorating SISTAG, 83--96, Contemp. Math., 314, Amer. Math. Soc., Providence, RI, 2002. 


\bibitem{Iarrobino}
A.~Iarrobino,
{\it Inverse system of symbolic power III. Thin algebras and fat points}, 
Compositio Math. 108 (1997), no. 3, 319--356. 

\bibitem{KeMc} S.~Keel and J.~McKernan, {\it Contraction of extremal rays on $\overline{M_{0,n}}$}, 
arXiv:alg-geom/9607009, 1996.



\bibitem{laface-ugaglia-TAMS}
A.~Laface and L.~Ugaglia,
{\it On a class of special linear systems on $\PP^3$}, Trans. Amer. Math. Soc. 358 (2006), no. 12, 5485--5500 (electronic). 

\bibitem{laface-ugaglia-BullBelg}
A. Laface and L.Ugaglia,
{\it On multiples of divisors associated to Veronese embeddings with defective secant variety},
Bull. Belg. Math. Soc. Simon Stevin 16 (2009), no. 5, Linear systems and subschemes, 933--942. 


\bibitem{laface-ugaglia-standard}
A.~Laface and L.~Ugaglia,
{\it Standard classes on the blow-up of $\PP^n$ at points in very general position}, arXiv:1004.4010 (2010).

\bibitem{ale-hirsch2} 
 E.~Postinghel, {\it A new proof of the Alexander-Hirschowitz interpolation theorem},
Ann. Mat. Pura Appl. (4) 191 (2012), no. 1, 77–-94. 

\bibitem{shafarevich-II}
I. R. Shafarevich, {\it Algebraic Geometry II. Cohomology of algebraic varieties. Algebraic surfaces},
Enc. Math. Sci. 35, Springer, Heidelberg 1996.


\end{thebibliography}
\end{document}